\documentclass[11pt,           
english,                       
reqno]{amsart}

\usepackage[english,strings]{babel}     
\usepackage{amsmath, amssymb}           
\usepackage{amsfonts}
\usepackage{amsthm}
\usepackage{theoremref}
\usepackage[utf8]{inputenc}    
\usepackage{longtable}         
\usepackage{exscale}           
\usepackage[final]{graphicx}   
\usepackage[sort]{cite}        
\usepackage{array}             
\usepackage{stmaryrd}
\usepackage{centernot}
\usepackage{mathtools}
\usepackage{enumerate}
\usepackage{wasysym}           
\usepackage{fancyhdr}          
\usepackage[a4paper]{geometry} 
\usepackage[multiuser]{fixme}  
\usepackage{xspace}            
\usepackage{tikz}              
\usetikzlibrary{matrix, arrows, cd}
\usepackage{ifdraft}           
\usepackage[expansion=false    
           ]{microtype}        
\usepackage[backref=page,      
           final=true,         
           pdfpagelabels       
           ]{hyperref}         
\usepackage[toc, page]
            {appendix}         

\graphicspath{{../tikz/}}

\allowdisplaybreaks

\usetikzlibrary{}

\FXRegisterAuthor{hb}{anhb}{henrique}
\FXRegisterAuthor{io}{anio}{inocencio}
\FXRegisterAuthor{sw}{answ}{stefan}

\ifdraft{\synctex=1}{}

\fxusetheme{color}

\newtheorem{theorem}{Theorem}
\newtheorem{definition}[theorem]{Definition}
\newtheorem{corollary}[theorem]{Corollary}
\newtheorem{lemma}[theorem]{Lemma}
\newtheorem{proposition}[theorem]{Proposition}
\newtheorem{defprop}[theorem]{Definition/Proposition}

\newtheorem{remark}[theorem]{Remark}
\newtheorem{problem}[theorem]{Problem}
\numberwithin{theorem}{section}

\newcommand{\Def}{\operatorname{Def}}
\newcommand{\Der}{\operatorname{Der}}
\newcommand{\FPois}{\operatorname{FPois}}

\newcommand{\End}{\operatorname{End}}

\newcommand{\Aut}{\operatorname{Aut}}

\newcommand{\Diff}{\operatorname{Diff}}

\newcommand{\ke}{\operatorname{ker}}

\newcommand{\Ann}{\operatorname{Ann}}
\newcommand{\Ve}{\operatorname{Ver}}

\newcommand{\alg}[1]{{\mathcal #1}}

\newcommand{\Cour}[1]      {\left[\!\left[#1\right]\!\right]}

\DeclareMathOperator{\Cinf}{\textsl{C}^{\hspace{0.5mm}\infty}}

\DeclareMathAlphabet\mathbfcal{OMS}{cmsy}{b}{n}

\newcommand{\CCE}{C_{\scriptscriptstyle{\mathrm{CE}}}}
\newcommand{\HCE}{H_{\scriptscriptstyle{\mathrm{CE}}}}
\newcommand{\CCEder}{C_{\scriptscriptstyle{\mathrm{CE, der}}}}
\newcommand{\HCEder}{H_{\scriptscriptstyle{\mathrm{CE, der}}}}
\newcommand{\HdR}{H_{\scriptscriptstyle{\mathrm{dR}}}}
\newcommand{\id}{\mathsf{id}}
\newcommand{\graph}{\operatorname{\mathrm{graph}}}
\newcommand{\canonical}{\mathrm{can}}

\newcommand{\1}{{\scriptscriptstyle{\mathrm{(1)}}}}
\newcommand{\2}{{\scriptscriptstyle{\mathrm{(2)}}}}
\newcommand{\ind}{{\scriptscriptstyle{(i)}}}
\begin{document}

%
%

\title{Morita equivalence of formal Poisson structures}

%
%

\author{Henrique Bursztyn}
\address{IMPA - Instituto Nacional de Matematica Pura e Aplicada \\
  Estrada Dona Castorina 110 \\
  Rio de Janeiro, 22460-320 \\
  Brazil}
\email{henrique@impa.br}

\author{Inocencio Ortiz}
\address{NIDTEC-FPUNA\\
  P.O.Box: 2111 SL\\
  CEP: 2160, San Lorenzo,  Paraguay.}
\email{inortiz08@gmail.com}.

\author{Stefan Waldmann}
\address{Department of Mathematics \\
  Julius Maximilian University of W\"urzburg \\
  Emil-Fischer-Straße 31 \\
  97074 Würzburg \\
  Germany}
\email{stefan.waldmann@mathematik.uni-wuerzburg.de}

%
%

\date{\today}

%
%

\selectlanguage{english}

%
%

\maketitle

%
%

\begin{abstract}
    We extend the notion of Morita equivalence of Poisson manifolds to
    the setting of {\em formal} Poisson structures, i.e., formal power
    series of bivector fields $\pi=\pi_0 + \lambda\pi_1 +\cdots$
    satisfying the Poisson integrability condition $[\pi,\pi]=0$. Our main result gives a
    complete description of Morita equivalent formal Poisson
    structures deforming the zero structure ($\pi_0=0$) in terms of
    $B$-field transformations, relying on a general study of formal
    deformations of Poisson morphisms and dual pairs. Combined with
    previous work on Morita equivalence of star products
    \cite{BDW-Characteristicclasses-starproducts}, our results link the notions of Morita
    equivalence in Poisson geometry and noncommutative algebra via
    deformation quantization.
\end{abstract}

%
%

\tableofcontents

%
%

\section{Introduction}
\label{sec:Introduction}

Poisson manifolds are often regarded as geometric analogues of (or ``first-order
approximations'' to) noncommutative algebras, as suggested by the physical
idea of {\em quantization}, and this principle has inspired important
advances in Poisson geometry (see e.g. \cite{CannasWeinstein,KorSoiBook} and references therein).
Along these lines, the notion of Morita equivalence \cite{Morita1958}, native to the theory of rings and algebras, has a geometric version for Poisson manifolds \cite{Xu1991}. Just as Morita equivalence of rings
is characterized by the existence of special types of bimodules, used
to establish equivalences of categories of representations, Morita
equivalence of Poisson manifolds is defined in terms of geometric
bimodules known as {\em dual pairs} \cite{Weinstein83}.  Although
these parallel Morita theories bear clear analogies (see
e.g. \cite{Bursztyn2005-PoissonGeomMoritaEquiv,Landsman200-BicategoriesOpAlgebras}), an explicit link between them has
been elusive.
The main purpose of this paper is to develop new aspects
of the Morita theory of Poisson structures in order to make such link
more transparent and tangible.

 More concretely, a way to relate Poisson structures to noncommutative algebras is via {\em deformation quantization} \cite{BayenFlatoetal-I,BayenFlatoetal-II}, a procedure that constructs
algebras of ``quantum observables'' on a given manifold
by means of formal associative deformations of its classical algebra of
smooth functions, called {\em star products}.
In this theory, Kontsevich \cite{Kontsevich2003} has shown that {\em formal} Poisson structures play a central role as the geometric counterparts of star products. The main motivating question behind our work is whether there is a precise sense in which deformation quantization relates Morita equivalences in Poisson geometry and noncommutative algebra. (The problem of relating Morita equivalences in different categories has been considered in other contexts as well, see e.g.  \cite{Landsman2001-OperatorAlgebras,Mrcun1999-Functoriality,clark2015equivalent,Muhlyetal87}.)
To address this question, this paper presents an extension of the geometric notion of Morita equivalence of Poisson manifolds to the formal setting. Our main results and how they relate to deformation quantization will be explained next.



\medskip

\noindent\textbf{Main results and outline of the paper.}
A {\em formal Poisson structure} on a manifold $P$ is a formal series of bivector fields $\pi = \sum_{j=0}^\infty \lambda^j \pi_j \in \mathfrak{X}^2(P)[[\lambda]]$ defining a Poisson bracket on the ring $C^\infty(P)[[\lambda]]$.
 Since $\pi_0$ is necessarily a Poisson structure, formal Poisson structures are naturally regarded as formal deformations of ordinary Poisson structures.
Particular examples are {\em formal symplectic structures}, i.e., formal series $\omega = \sum_{j=0}^\infty \lambda^j \omega_j$ of closed 2-forms with $\omega_0$
symplectic; in this case, the nondegeneracy of $\omega_0$ implies that $\omega$ can be formally inverted to define
a formal Poisson structure.


We introduce Morita equivalence of formal Poisson structures in Section~\ref{sec:Preliminaries} as a deformation of the original notion of Morita equivalence for Poisson manifolds from \cite{Xu1991}. Given smooth manifolds $P_1$ and $P_2$ carrying formal Poisson
structures $\pi^\1 $ and $\pi^\2$, respectively, their Morita equivalence is defined by the existence
of a \emph{formal equivalence bimodule}, which consists of a formal symplectic
manifold $(S,\omega=\sum_{j=0}^\infty \lambda^j \omega_j)$ fitting
into a diagram
\begin{equation}
    \label{eq:formalEquiv}
    (\Cinf(P_1)[[\lambda]], \pi^\1)
    \stackrel{\Phi^\1}{\longrightarrow}
    (\Cinf(S)[[\lambda]], \omega)
    \stackrel{\Phi^\2}\longleftarrow
    (\Cinf(P_2)[[\lambda]], \pi^\2),
\end{equation}
where $\Phi^\1=\sum_{j=0}^\infty \lambda^j \Phi^\1_j$ (resp. $\Phi^\2 = \sum_{j=0}^\infty \lambda^j \Phi^\2_j$) is a Poisson (resp. anti-Poisson) map, their images Poisson commute in $\Cinf(S)[[\lambda]]$, and the underlying zeroth order diagram (obtained by setting $\lambda=0$),
\begin{equation*}
    (P_1,\pi^\1_0)
    \xleftarrow[]{J_1}
    (S, \omega_0)
    \xrightarrow[]{J_2}
    (P_2,\pi^\2_0),
\end{equation*}
is an equivalence bimodule (hence defines a Morita equivalence) in the ordinary sense of \cite{Xu1991}.
Here, for $i=1,2$, $J_i\colon S \to P_i$ is the classical map such that $J_i^*= \Phi^\ind_0\colon C^\infty(P_i)\to
C^\infty(S)$. With this definition in place, our goal is to describe Morita equivalence within the subset of formal Poisson structures on a manifold $P$ which vanish in zeroth order.

A key ingredient to formulate our main result is the notion of \emph{gauge transformation} of Poisson structures \cite{SeveraWeinstein2001}, also called
\emph{$B$-field} transforms \cite{Gualtieri2011}. Given a
Poisson manifold $(P,\pi)$ and a closed 2-form $B\in \Omega^2(P)$,
consider the associated bundle maps $\pi^\sharp\colon T^*P\to TP$ and
$B^\flat\colon TP\to T^*P$, and assume that $\pi$ and $B$ are
\emph{compatible} in the sense that $\id+ B^\flat\pi^\sharp\colon T^*P
\to T^*P$ is an isomorphism. In this case, the gauge transformation of
$\pi$ by $B$ is a new Poisson structure $\tau_B(\pi)$ defined by
\begin{equation*}
    {(\tau_B(\pi))}^\sharp
    =
    \pi^\sharp (\id+ B^\flat\pi^\sharp)^{-1},
\end{equation*}
see also \cite{BursztynFernandes-Picard,BursztynRadko2003}.
In the formal setting, we have a similar picture: if $\pi=\sum_{j=0}^\infty \lambda^j \pi_j$
is a formal Poisson structure and $B=\sum_{j=0}^\infty \lambda^j B_j$
is a formal series of closed 2-forms on $P$, assuming that $\pi_0$ and
$B_0$ are compatible, we obtain a new formal Poisson structure
$\tau_B(\pi)$.

When we restrict our attention to formal Poisson structures on $P$
vanishing in zeroth order, gauge transformations are well-defined for any
closed $B = \sum_{j=0}^\infty \lambda^j B_j$ (since the bivector
$\pi_0=0$ is compatible with any $B_0\in \Omega^2(P)$). Denoting by $\FPois_0(P)$ the set of equivalence classes of formal Poisson structure vanishing in zeroth order (modulo formal diffeomorphisms) and by $\HdR^2(P)$ the second de Rham cohomology, it was shown in
\cite{BDW-Characteristicclasses-starproducts} that there is an induced action
\begin{equation}
    \label{eq:Bfieldact}
    \HdR^2(P)[[\lambda]] \times \FPois_0(P)
    \to
    \FPois_0(P),
    \qquad ([B],[\pi])\mapsto [\tau_B(\pi)],
\end{equation}
where $\HdR^2(P)[[\lambda]]$ is viewed as an additive
group. Our main result fully characterizes Morita equivalence in terms of
gauge transformations:
\begin{theorem}
    \label{thm:main}%
    Two formal Poisson structures $\pi$ and $\pi'$ on $P$, vanishing
    in zeroth order, are Morita equivalent if and only if there exists
    a diffeomorphism $\psi \in \mathrm{Diff}(P)$ such that $\pi$ and
    $\psi_*\pi'$ lie in the same orbit of the action of $\HdR^2(P)[[\lambda]]$  on
    $\FPois_0(P)$ by gauge transformations.
\end{theorem}
Hence Morita equivalence and gauge
transformations of Poisson structures coincide (modulo diffeomorphisms) in this formal context, in contrast with the
classical setting (cf. \cite[Sec.~5]{BursztynRadko2003}).

The core of the paper is devoted to the proof of Theorem~\ref{thm:main}, and we briefly outline its main ingredients. From the very definition of equivalence bimodules \eqref{eq:formalEquiv} for formal Poisson structures, it is clear that their existence leads to the following natural deformation problem:

\begin{problem}\label{prob} Given a classical Morita equivalence
\begin{equation}\label{bimod}
    (P_1,\pi_1)
    \xleftarrow[]{J_1}
    (S, \omega_0)
    \xrightarrow[]{J_2}
    (P_2,\pi_2),
\end{equation}
a formal Poisson structure $\pi^\2=\pi_2 + \sum_{j=1}^\infty\lambda^j \pi^\2_j$ and a formal symplectic structure
$\omega=\omega_0 + \sum_{j=1}^\infty \lambda^j\omega_j$,
can one find formal deformations $\Phi^\2 = J_2^* + \sum_{j=1}^\infty \lambda^j \Phi^\2_j$, $\pi^\1=\pi_1 + \sum_{j=1}^\infty\lambda^j \pi^\1_j$ and $\Phi^\1 = J_1^* + \sum_{j=1}^\infty \lambda^j \Phi^\1_j$ defining an equivalence bimodule as in \eqref{eq:formalEquiv}?
\end{problem}

We begin the analysis of this problem in Section \ref{sec:defmorphism} by dividing it into two
parts, both treated in a completely algebraic framework.

\begin{itemize}
\item We first consider deformations of Poisson morphisms. Given Poisson algebras $\mathcal{A}$ and $\mathcal{B}$, with Poisson brackets
$\pi_0$ and $\sigma_0$, and a Poisson morphism $\phi_0:
(\mathcal{A},\pi_0) \to (\mathcal{B},\sigma_0)$, if we fix formal
Poisson structures $\pi=\pi_0 + \sum_{j=1}^\infty \lambda^j \pi_j$ and
$\sigma=\sigma_0+ \sum_{j=1}^\infty \lambda^j \sigma_j $, we consider the problem of finding a
Poisson morphism
\begin{equation}
    \label{eq:Phi}
    \Phi\colon
    (\mathcal{A}[[\lambda]],\pi) \to (\mathcal{B}[[\lambda]],\sigma),
\end{equation}
with $\Phi=\phi_0 + \sum_{j=1}^\infty \lambda^j \phi_j$. We identify the cohomologies governing
this deformation problem and describe existence and uniqueness results
in Propositions~\ref{Existence} and \ref{Uniqueness}.

\item The second part concerns commutants. Given a Poisson
map as in \eqref{eq:Phi}, let $\mathcal{A}'$ be the Poisson commutant
of $\phi_0(\mathcal{A})$ in $(\mathcal{B},\sigma_0)$ and $C$ be the
Poisson commutant of $\Phi(\mathcal{A}[[\lambda]])$ in
$(\mathcal{B}[[\lambda]],\sigma)$; the problem is then deforming the
inclusion $\mathcal{A}'\to \mathcal{B}$ into an isomorphism (of
commutative rings) $\Phi': \mathcal{A}'[[\lambda]] \to C\subseteq
\mathcal{B}[[\lambda]]$. Since $C$ has a natural Poisson structure,
this isomorphism induces a formal Poisson structure $\pi'$ on
$\mathcal{A}'[[\lambda]]$ deforming that of $\mathcal{A}'$, leading to
the following diagram of Poisson maps with Poisson commuting images:
\begin{equation}\label{eq:Comm}
    (\mathcal{A}[[\lambda]],\pi)
    \xrightarrow[]{\Phi}
    (\mathcal{B}[[\lambda]],\sigma)
    \xleftarrow[]{\Phi'}
    (\mathcal{A}'[[\lambda]],\pi').
\end{equation}
Conditions for finding $\Phi'$ and the uniqueness properties of
the  resulting formal Poisson structure $\pi'$ are presented in
Propositions~\ref{Deformationofthecommutant} and
\ref{prop:unique}.
\end{itemize}

In Section~\ref{sec:ClassifyingMap}, we return to the geometric
setting of Problem \ref{prob} but focusing on
formal Poisson structures that vanish in zeroth order. With this additional assumption, the classical bimodule \eqref{bimod} is a self-equivalence of the trivial Poisson manifold $(P,\pi_0=0)$, and those have been proven in
\cite{BursztynFernandes-Picard, BursztynWeinstein-Picard} to be of the form $S=T^*P$ and $\omega_0 = \omega_{\mathrm{can}} +
\rho^*B_0$, where $\omega_{\mathrm{can}}$ is the canonical symplectic form on
$T^*P$, $\rho: T^*P\to P$ is the natural projection, and $B_0$ is a closed 2-form on $P$.
With this description and a geometric interpretation of
the cohomological conditions arising in \eqref{eq:Phi} and \eqref{eq:Comm}, we
prove in Theorem~\ref{DeformationofthezeroPoissonsymplecticrealization} that, in this case,
the deformations in Problem \ref{prob} are unobstructed and unique, in a natural
sense. We show, as a consequence, that for any given formal Poisson structure $\pi = \lambda\pi_1 + \cdots$ on $P$ and $B\in \Omega^2(P)[[\lambda]]$ closed, we
obtain a formal equivalence bimodule
\begin{equation*}
    (\Cinf(P)[[\lambda]], \pi)
    \stackrel{\Phi}{\longrightarrow}
    (\Cinf(T^*P)[[\lambda]], \omega_B)
    \stackrel{\Phi'}{\longleftarrow}
    (\Cinf(P)[[\lambda]], \pi^B),
\end{equation*}
where $\omega_B=\omega_{\mathrm{can}}+\rho^* B$, $\Phi$ and $\Phi'$ are
deformations of $\rho^*$, and $\pi^B$
a formal Poisson structure on $P$ (vanishing in zeroth order) determined by $\pi$ and $B$ (cf. \eqref{eq:Comm}).
This construction leads to a map
\begin{equation}
    \label{eq:classifyingmap}
\gamma:    \HdR^2(P)[[\lambda]] \times \FPois_0(P)
    \rightarrow
    \FPois_0(P),
    \qquad
    ([B], [\pi])\mapsto [\pi^B],
\end{equation}
defining an action that completely characterizes Morita equivalence of formal Poisson structures in $\FPois_0(P)$, as explained in Theorem~\ref{CharacterizationofMEFPSviadeformation}.
We call it the \emph{classifying action}.

In Section~\ref{sec:B-field_classification}, we complete the proof of
Theorem~\ref{thm:main} by showing, with tools from \cite{Frejlich}, that the classifying action agrees with
the action \eqref{eq:Bfieldact} of $B$-fields on formal Poisson structures.


\medskip

\noindent\textbf{Link with deformation quantization}.
We now explain how Theorem~\ref{thm:main}, combined with other results in the literature, allows us to establish a concrete link between Morita equivalences in Poisson geometry and algebra through deformation quantization.

A {\em star product} on a manifold $P$ is a formal associative deformation of the algebra of ($\mathbb{C}$-valued) smooth functions on $P$, i.e., a product on $C^\infty(P)[[\lambda]]$ of the form
$f\star g = fg + \sum_{k=1}^\infty\lambda^k C_k(f,g),$
where each $C_k: C^\infty(P)\times C^\infty(P)\to C^\infty(P)$ is a bidifferential operator. Two star products are {\em equivalent} if they are isomorphic via $\id + \sum_{k=1}^\infty \lambda^k T_k$, for differential operators $T_k: C^\infty(P)\to C^\infty(P)$; we denote the set of equivalence classes of star products on $P$ by $\Def(P)$.
Poisson geometry enters the picture through the fact that any star product $\star$ {\em quantizes} a Poisson structure on $P$ given by the semi-classical limit of its commutators: $\{f,g\}:= \frac{1}{\lambda}(f\star g - g\star f) |_{\lambda=0}$.

A celebrated result of Kontsevich \cite{Kontsevich2003} asserts that there are as many classes of star products quantizing a given Poisson structure on $P$ as there are classes of formal Poisson deformations of this Poisson structure; more precisely, there is a bijective correspondence
\begin{equation}
    \label{eq:K}
    \mathcal{K}\colon
    \mathrm{FPois}_0(P) \stackrel{\sim}{\to} \mathrm{Def}(P),
\end{equation}
 with the property that star products in the class $\mathcal{K}(\pi)$ quantize the Poisson structure $\pi_1$ such that $\pi = \lambda\pi_1 + \cdots$. Hence, once geometric Morita equivalence is extended to formal Poisson structures, it makes sense to use  $\mathcal{K}$ to compare it with algebraic Morita equivalence of star products.

To tackle this problem, recall that the
Morita equivalence classes of star products on  $P$ are characterized as orbits of a natural action of the group $\mathrm{Diff}(P)\ltimes H^2(P,\mathbb{Z})$ on $\Def(P)$ \cite[Theorem~4.1]{Bursztyn-QuantumLineBundles}.
The main contribution of this paper is to show a similar picture for formal Poisson structures (see Theorem~\ref{CharacterizationofMEFPSviadeformation}): Morita equivalence classes in $\FPois_0(P)$ are orbits of an action of $\mathrm{Diff}(P)\ltimes \HdR^2(P,\mathbb{C})[[\lambda]]$; additionally, Theorem~\ref{thm:main} gives an explicit description of this action in terms of gauge transformations.
As a final ingredient,  \cite[Theorem~3.11]{BDW-Characteristicclasses-starproducts} relates these results by asserting that
the map \eqref{eq:K} is $(\mathrm{Diff}(P)\ltimes H^2(P,\mathbb{Z}))$-equivariant, i.e., it
intertwines gauge transformations on $\FPois_0(P)$ by 2-forms in the image of the
natural map $2\pi i H^2(P,\mathbb{Z})\to \HdR^2(P,\mathbb{C})$ with the action on $\Def(P)$. The conclusion is that, under Kontsevich's quantization map $\mathcal{K}$, Morita
equivalence of formal Poisson structures by $B$-fields in $2\pi i H^2(P,\mathbb{Z})$
corresponds to Morita equivalence of star products, so the notions coincide upon an integrality condition.

\medskip

\noindent\textbf{Acknowledgments}. H. Bursztyn and I. Ortiz thank
CNPq, Faperj, INCTMat and FPUNA for financial support.  Several institutions
have hosted us during various stages of this project, including IMPA,
U. W\"urzburg, U. Buenos Aires, Erwin Sch\"odinger Institute and UFRJ. We have benefited from discussions with A. Cabrera and R. L. Fernandes.

%
%

\section{Morita equivalence of formal Poisson structures}
\label{sec:Preliminaries}

We start by recalling some definitions and setting up the notation used throughout the paper. Smooth functions and tensors on a manifold $P$ will be considered with ground field $\mathbb{K}=\mathbb{R}$ or $\mathbb{C}$.

\subsection{Preliminaries}\label{subsec:prelim}
A Poisson structure on a manifold $P$ will be denoted by either a Poisson bivector field $\pi \in \mathfrak{X}^2(P)$ or by its corresponding Poisson bracket $\{\cdot,\cdot\}$ on $C^\infty(P)$,
$\{f,g\}=\pi(df,dg)$.
Given $f\in C^\infty(P)$ its hamiltonian vector field is $X_f = \pi^\sharp(df) = i_{df}\pi$, so that $\mathcal{L}_{X_f}g = \{f,g\}$.
A Poisson map $\varphi: (P_1,\pi_1)\to (P_2,\pi_2)$ is {\bf complete} if, whenever $X_f$ is a complete vector field for $f\in C^\infty(P_2)$, then so is $X_{\varphi^*f}$. A map $\varphi: (P_1,\pi_1)\to (P_2,\pi_2)$ is {\bf anti-Poisson} if
$\varphi: (P_1,\pi_1)\to (P_2,-\pi_2)$ is Poisson.

We denote by $C^\infty(P)[[\lambda]]$ the space of formal power series in $\lambda$ with coefficients in $C^\infty(P)$,
and we use similar notation when $C^\infty(P)$ is replaced by the space of multivector fields $\mathfrak{X}^\bullet(P)$ or differential forms $\Omega^\bullet(P)$; these spaces will be always regarded as modules over $\mathbb{K}[[\lambda]]$. We consider the ($\lambda$-linear) extensions of the de Rham differential to $\Omega^\bullet(P)[[\lambda]]$ and Schouten bracket to $\mathfrak{X}^\bullet(P)[[\lambda]]$.

For a formal vector field $X = \sum_{j=1}^{\infty}\lambda^jX_j \in \lambda \mathfrak{X}(P)[[\lambda]]$, the formal series
\begin{equation}
    \label{eq:formaldiffeo}
    \exp(\mathcal{L}_X)
    =
    \id + \mathcal{L}_X+\frac{(\mathcal{L}_X)^2}{2!} + \cdots
\end{equation}
is called a {\bf formal diffeomorphism} on $P$, where $\mathcal{L}_X= \sum_{j=1}^{\infty}\lambda^j\mathcal{L}_{X_j}$ is the Lie derivative along $X$. Formal diffeomorphisms form a group (thanks to the
Baker-Campbell-Hausdorff formula) that naturally acts on $C^\infty(P)[[\lambda]]$, $\Omega^\bullet(P)[[\lambda]]$ and $\mathfrak{X}^\bullet(P)[[\lambda]]$ preserving their ring structures (as well as the de Rham differential on $\Omega^\bullet(P)[[\lambda]]$ and Schouten bracket on $\mathfrak{X}^\bullet(P)[[\lambda]]$).

 For manifolds $P$ and $S$, consider $C^\infty(P)[[\lambda]]$ and $C^\infty(S)[[\lambda]]$ with their commutative products. Any morphism of commutative rings $\Phi: C^\infty(P)[[\lambda]]\to C^\infty(S)[[\lambda]]$ is a formal series $\Phi=\sum_{j=0}^\infty \lambda^j \phi_j$ of linear maps $\phi_j: C^\infty(P)\to C^\infty(S)$. Moreover, $\phi_0: C^\infty(P)\to C^\infty(S)$ is a morphism of algebras, hence of the form $\phi_0=J^*$ for a smooth map $J:S \to P$. If $S=P$ and $\phi_0=\id$, then $\Phi=\exp(\mathcal{L}_X)$ for a formal vector field $X$.

\begin{lemma}\label{lem:formalVF}
Let $\Phi= J^* + \sum_{k=1}^\infty \lambda^k \phi_k: C^\infty(P)[[\lambda]]\to C^\infty(S)[[\lambda]]$ be a morphism of commutative rings so that $J: S\to P$ is a surjective submersion.
Then there exists a formal vector field $Z\in \lambda \mathfrak{X}(S)[[\lambda]]$ such that $\Phi=\exp(\mathcal{L}_Z) J^*$.
\end{lemma}

\begin{proof}
First recall that any linear map $\phi: C^\infty(P)\to C^\infty(S)$ which is a derivation along $J^*$, i.e., which satisfies $\phi(fg) = \phi(f) J^*g + J^*f \phi(g)$, is an element in $\Gamma(J^*TP)$. If $J$ is a surjective submersion, by considering a horizontal lift $\Gamma(J^*TP)\to \mathfrak{X}(S)$, we can find a vector field $Z\in \mathfrak{X}(S)$ with $\phi = \mathcal{L}_Z \circ J^*$.

By expanding in $\lambda$ the condition $\Phi(fg) = \Phi(f)\Phi(g)$ for all
    $f, g\in\Cinf(P)$, we obtain
    \begin{equation}
        \label{eqforphik}
        \phi_k(fg)
        =
        \sum_{i+j=k}\phi_i(f)\phi_j(g),
    \end{equation}
    with $\phi_0 = J^*$. Note that, for $k=1$, we have that $\phi_1$ is a derivation along $J^*$, so we can find $Z_1 \in \mathfrak{X}(S)$ such that $\phi_1=\mathcal{L}_{Z_1} \circ J^*$. For $Z_{(1)}=\lambda Z_1$,
    it follows that $\Phi$ agrees with $\exp(\mathcal{L}_{Z_{(1)}}) J^*$ modulo $\lambda^2$.
    Suppose now that we have vector fields $Z_1,\ldots,Z_{k-1} \in \mathfrak{X}(S)$ such that $\Phi$ agrees with $\exp(\mathcal{L}_{Z_{(k-1)}}) J^*$ modulo $\lambda^k$, where $Z_{(k-1)}= \lambda Z_1 + \ldots + \lambda^{k-1}Z_{k-1}$. Denoting the $k^{th}$ order term of $\exp(\mathcal{L}_{Z_{(k-1)}}) J^*$ by $E_k$, equation \eqref{eqforphik} implies that $\phi_k-E_k$ is a a derivation along $J^*$, so there is a vector field $Z_k$ on $S$ such that $\phi_k-E_k = \mathcal{L}_{Z_k}\circ J^*$. Setting $Z_{(k)}= \lambda Z_1 + \ldots + \lambda^{k}Z_{k}$, one directly checks that $\phi_k = E_k+ \mathcal{L}_{Z_k}\circ J^*$ agrees with $k^{th}$ order term of $\exp(\mathcal{L}_{Z_{(k)}}) J^*$, and hence $\Phi$ agrees with $\exp(\mathcal{L}_{Z_{(k)}}) J^*$ modulo $\lambda^{k+1}$, so the result follows by induction.

\end{proof}

On a manifold $P$,
 a \textbf{formal Poisson structure} is a formal series  $\pi =\sum_{j=0}^\infty\lambda^j\pi_j\in\mathfrak{X}^2(P)[[\lambda]]$ such that $[\pi,\pi]=0$, where $[\cdot,\cdot]$ is the Schouten bracket (extended to formal bivector fields $\lambda$-bilinearly). It is clear from the integrability equation that, if $\pi_j=0$ for $j<k$, then $\pi_k$ is an ordinary Poisson structure.
Just as in the ordinary setting, a formal Poisson structure can be  viewed as a Poisson bracket on the ring $C^\infty(P)[[\lambda]]$.

 Two formal Poisson structures $\pi^\1$ and $\pi^\2$ on $P$ are \textbf{equivalent}
    if there exists a formal vector field
    $X\in\lambda\mathfrak{X}(P)[[\lambda]]$ such that $\pi^\2 =
    \exp(\mathcal{L}_X)\pi^\1$.
In terms of the corresponding Poisson brackets $\{\cdot,\cdot\}^{(i)}$, $i=1,2$, this amounts to saying that
$$
    \exp(\mathcal{L}_X)\colon
    (\Cinf(P)[[\lambda]], \{\cdot,\cdot\}^\1)
    \to
    (\Cinf(P)[[\lambda]], \{\cdot,\cdot\}^\2)
$$
is bracket preserving. Note that if $\pi^\1$ and $\pi^\2$ are equivalent, then they agree in zeroth order: $\pi^\1_0=\pi^\2_0$. For a given Poisson structure $\pi_0$ on $P$, we denote by $\FPois_{\pi_0}(P)$ the set of formal Poisson structures on $P$ deforming $\pi_0$, up to equivalence.

A formal Poisson structure $\pi$ on $P$ also gives rise to a linear map
\begin{equation}\label{eq:pisharp}
\pi^\sharp\colon \Omega^1(P)[[\lambda]]\to \mathfrak{X}(P)[[\lambda]],\;\;\; \alpha\mapsto i_\alpha\pi.
\end{equation}
Given $f\in C^\infty(P)[[\lambda]]$, its hamiltonian vector field is defined as the formal vector field $\pi^\sharp(df)$.

A special class of formal Poisson structure is given by {\bf formal symplectic structures}, i.e.,
formal series of closed 2-forms $\omega = \omega_0 +\sum_{k=1}^\infty\lambda^k\omega_k$ with $\omega_0$
 symplectic. In this case the corresponding map
$$
\omega^\flat\colon \mathfrak{X}(P)[[\lambda]]\to \Omega^1(P)[[\lambda]], \;\;\; X\mapsto i_X\omega,
$$
can be formally inverted, since in zeroth order it is given by the invertible map $\omega_0^\flat\colon \mathfrak{X}(P) \to \Omega^1(P)$, and its inverse is a map \eqref{eq:pisharp} determining a formal Poisson structure on $P$. More generally, if
$\omega_0$ is a symplectic form on $P$, with corresponding Poisson bivector field $\pi_0$, then this formal inversion establishes a bijective correspondence between formal Poisson deformations of $\pi_0$ and formal symplectic forms $\omega_0 +\sum_{k=1}^\infty\lambda^k\omega_k$.

\begin{remark}\label{rem:symplecticdeformations}
   Under this correspondence, equivalent formal Poisson structures are identified with cohomologous formal symplectic forms, so when $\pi_0$ is symplectic we have $\FPois_{\pi_0}(P) \cong \lambda \HdR^2(P)[[\lambda]]$ \cite{Lecomte-LMP13-1987} (see also \cite[Prop.~13]{Gutt-variationsonDQ-2000})
\end{remark}

A Poisson map $\Phi\colon (C^\infty(P_1)[[\lambda]],\pi^\1)\to (C^\infty(P_2)[[\lambda]],\pi^\2)$ is always given by a formal series $\Phi=\sum_{k=0}^\infty \lambda^k \phi_k$, with $\phi_k\colon C^\infty(P_1)\to C^\infty(P_2)$; the map $\phi_0$ is necessarily a morphism of Poisson algebras,
$$
\phi_0\colon (C^\infty(P_1),\{\cdot,\cdot\}^\1_0)\to (C^\infty(P_2),\{\cdot,\cdot\}^\2_0),
$$
hence must be of the form $\phi_0 = \varphi^*$ for a Poisson map $\varphi\colon (P_2,\pi^\2_0)\to (P_1,\pi^\1_0)$.

\subsection{Morita equivalence of Poisson manifolds}
Given Poisson manifolds $(P_i, \pi_i)$, $i = 1, 2$, let us consider a diagram of the form
\begin{equation}\label{eq:bimodule}
        (P_1, \pi_1)
        \stackrel{J_1}{\longleftarrow}
        (S, \omega_0)
        \stackrel{J_2}\longrightarrow
        (P_2, \pi_2),
\end{equation}
where $(S, \omega_0)$ is a symplectic manifold, $J_1$ is a Poisson map, and $J_2$ is an anti-Poisson map.
We will call it a $(P_1,P_2)$-{\bf bimodule} if the subalgebras $J_i^*C^\infty(P_i)$, $i=1,2$, Poisson commute in $C^\infty(S)$.

The notion of Morita equivalence of Poisson manifolds, introduced in \cite{Xu1991}, relies on special types of bimodules, satisfying additional regularity conditions:

\begin{definition}\label{def:EBim}
An {\bf equivalence bimodule} is a diagram as in \eqref{eq:bimodule} such that the maps $J_1$ and $J_2$ are surjective submersions, complete, with connected and simply-connected fibers, and the subbundles tangent to their fibers are symplectic orthogonal complements of each other.
\end{definition}

Recall that, for a Poisson algebra $\mathcal{A}$ with Poisson subalgebra $\mathcal{B}\subseteq \mathcal{A}$, the {\bf commutant} of $\mathcal{B}$ in $\mathcal{A}$ is the Poisson subalgebra $\mathcal{B}^c:=\{a\in \mathcal{A}\,|\, \{a,\mathcal{B}\}=0 \}$.
For equivalence bimodules, the Poisson subalgebras $J_i^*C^\infty(P_i)\subseteq C^\infty(S)$, $i=1,2$, are commutants of one another (see \cite{MOR2004-LocalGlobarPair}); these bimodules are special cases of  the ``dual pairs'' of \cite[Sec.~8]{Weinstein83}).

\begin{definition}\label{definition:MoritaPoisson}
    Two Poisson manifolds $(P_1, \pi_1)$ and $(P_2,\pi_2)$
    are \textbf{Morita equivalent} if there is an equivalence bimodule
    \begin{equation}\label{eq:Mequiv}
        (P_1, \pi_1)
        \stackrel{J_1}{\longleftarrow}
        (S, \omega_0)
        \stackrel{J_2}\longrightarrow
        (P_2, \pi_2).
    \end{equation}
     \end{definition}

Not every Poisson manifold can be part of an equivalence bimodule; as shown in \cite{Crainicetal2004}, this can only happen if the Poisson manifold is {\em integrable} (in the sense that it admits an integration by a symplectic groupoid). But within integrable Poisson manifolds, Morita equivalence does define an equivalence relation \cite{Xu1991,Xu1991-MESymplecticGroupoids}. For more on Poisson Morita equivalence, see e.g. \cite{BursztynWeinstein-Picard}.

We now pass to the formal context.

\subsection{Morita equivalence in the formal setting} \label{subsec:formal}
Let $(P_i, \pi^\ind)$,  $i = 1, 2$, be formal Poisson manifolds. As in the classical case, a {\bf bimodule}
is a diagram
\begin{equation}\label{eq:formalBim}
        (\Cinf(P_1)[[\lambda]], \pi^\1)
        \stackrel{\Phi^\1}{\longrightarrow}
        (\Cinf(S)[[\lambda]], \omega)
        \stackrel{\Phi^\2}\longleftarrow
        (\Cinf(P_2)[[\lambda]], \pi^\2),
\end{equation}
where $\omega=\sum_{k=0}^\infty \lambda^k \omega_k$ is a formal symplectic structure on $S$,
$\Phi^\1$ is a Poisson morphism, $\Phi^\2$ is an anti-Poisson morphism, and the Poisson subalgebras $\Phi^\1(\Cinf(P_1)[[\lambda]])$ and $\Phi^\2(\Cinf(P_2)[[\lambda]])$ Poisson commute in $(\Cinf(S)[[\lambda]], \omega)$.

Setting $\lambda=0$, we obtain a geometric diagram
\begin{equation}\label{eq:classdualpair}
(P_1, \pi^\1_0)
        \stackrel{J_1}{\longleftarrow}
        (S, \omega_0)
        \stackrel{J_2}\longrightarrow
        (P_2, \pi^\2_0),
\end{equation}
where $J_1$ (resp. $J_2$) is a Poisson (resp. anti Poisson) map agreeing with $\Phi^\1$ (resp. $J_2^*=\Phi^\2$) in zeroth order.

\begin{definition}
    \label{definition:FormalPoissonME}%
A bimodule as in \eqref{eq:formalBim} is a (formal) {\bf  equivalence bimodule} if its underlying geometric diagram
\eqref{eq:classdualpair} is an equivalence bimodule in the sense of Definition~\ref{def:EBim}.
\end{definition}

Two formal Poisson manifolds  $(P_1,\pi^\1)$ and $(P_2,\pi^\2)$  are {\bf Morita equivalent} if they fit into an equivalence bimodule as in \eqref{eq:formalBim}. In particular, the zeroth order Poisson manifolds $(P_1,\pi^\1_0)$ and $(P_2,\pi^\2_0)$ are Morita equivalent in the ordinary sense.

As in the original setting, Morita equivalence is not defined for all formal Poisson structures, and it seems a difficult problem to characterize the subclass where Morita equivalence defines an equivalence relation (the integrability of the Poisson structures in zeroth order is clearly necessary, but our results in Section~\ref{sec:defmorphism} identify obstructions indicating that this condition is not enough). Within formal Poisson structures vanishing in zeroth order, which is the focus of this paper, Theorem~\ref{thm:main} ensures that Morita equivalence is a well defined equivalence relation.

Making use of Lemma~\ref{lem:formalVF}, we can rephrase the definition of Morita equivalence of formal Poisson structures as follows.

\begin{defprop}
Two formal Poisson manifolds $(P_1,\pi^\1)$ and $(P_2,\pi^\2)$ are Morita equivalent if and only if there is an equivalence bimodule
$$
(P_1, \pi^\1_0)
        \stackrel{J_1}{\longleftarrow}
        (S, \omega_0)
        \stackrel{J_2}\longrightarrow
        (P_2, \pi^\2_0),
$$
along with a formal symplectic form $\omega = \omega_0 + \sum_{k=1}^\infty \lambda^k \omega_k$ on $S$ and formal vector fields $Z^{(i)}\in \lambda\mathfrak{X}(S)[[\lambda]]$ so that $\Phi^\1 = \exp(\mathcal{L}_{Z^\1})J_1^*\colon (C^\infty(P_1)[[\lambda]],\pi^\1)\to (C^\infty(S)[[\lambda]],\omega)$ is a Poisson morphism, $\Phi^\2 = \exp(\mathcal{L}_{Z^\2})J_2^*\colon (C^\infty(P_2)[[\lambda]],\pi^\2)\to (C^\infty(S)[[\lambda]],\omega)$ is an anti-Poisson morphism, and
\begin{equation}
        \label{altPoissoncommutation}
        \big\{
        \Phi^\1(\Cinf(P_1)[[\lambda]]),
        \Phi^\2(\Cinf(P_2)[[\lambda]])
        \big\}_{\omega}=0.
    \end{equation}
\end{defprop}


As in the geometrical case, equivalence bimodules in the formal setting have the following property.

\begin{proposition}\label{prop:commimage}
Let
$$
        (\Cinf(P_1)[[\lambda]], \pi^\1)
        \stackrel{\Phi^\1}{\longrightarrow}
        (\Cinf(S)[[\lambda]], \omega)
        \stackrel{\Phi^\2}\longleftarrow
        (\Cinf(P_2)[[\lambda]], \pi^\2)
$$
be an equivalence bimodule. Then the subalgebras $\Phi^\1(\Cinf(P_1)[[\lambda]])$ and $\Phi^\2(\Cinf(P_2)[[\lambda]])$ are commutants of one another in $(\Cinf(S)[[\lambda]], \omega)$.
\end{proposition}

\begin{proof}
Let us write $\Phi^\1 = \exp(\mathcal{L}_{Z^\1})J_1^*$ and $\Phi^\2 = \exp(\mathcal{L}_{Z^\2})J_2^*$.
We will use the fact that $J_1^*(C^\infty(P_1))$ and $J_2^*(C^\infty(P_2))$ are commutants of one another in $(C^\infty(S),\omega_0)$ and property \eqref{altPoissoncommutation} to show the result.

Let
    $A_i= \Cinf(P_i)[[\lambda]]$, for $i = 1, 2$, and let $A_i'$ be
    the Poisson commutant of $\Phi^\ind(A_i)$ inside
    $(\Cinf(S)[[\lambda]], \omega)$. Note that condition
    \eqref{altPoissoncommutation} means that
    $\exp(\mathcal{L}_{Z^\1})$ maps $J_1^*(A_1)$ (injectively) into
    $A_2'$. Now let $F = F_0+\lambda F_1+\cdot\cdot\cdot\in A_2'$.
    Then $\{F, \Phi^\2(g)\}_\omega = 0$ for all
    $g\in\Cinf(P_2)$, which in zeroth order means that
    $\{F_0, J_2^*(g)\}_{\omega_0} = 0$ for all $g\in \Cinf(P_2)$. It follows that  $F_0 = J_1^*f_0$, for some $f_0\in\Cinf(P_1)$. and hence
    $F-\exp(\mathcal{L}_{Z^\1})J_1^*f_0 =
    \lambda\hat{F_1}+\cdot\cdot\cdot\in A_2'$. Repeating the argument, we see that $\hat{F_1} = J_1^*f_1$, with $f_1\in\Cinf(P_1)$. By iterating this
    argument, we conclude that $F = \exp(\mathcal{L}_{Z^\1})J_1^*f$, for some
    $f\in A_1$. Hence, $\Phi^\1(A_1) = A_2'$, and by symmetry we also
    have that $\Phi^\2(A_2) = A_1'$.
\end{proof}

\subsection{Morita equivalence and \emph{B}-fields}\label{subsec:MEBfields}

Given an equivalence bimodule
\begin{equation}\label{eq:EB}
(\Cinf(P_1)[[\lambda]], \pi^\1)
        \stackrel{\Phi^\1}{\longrightarrow}
        (\Cinf(S)[[\lambda]], \omega)
        \stackrel{\Phi^\2}\longleftarrow
        (\Cinf(P_2)[[\lambda]], \pi^\2),
\end{equation}
one can naturally modify it  by formal Poisson diffeomorphisms of $P_1$, $P_2$ or $S$. For example, a formal diffeomorphism $\exp(\mathcal{L}_Z)$ on $S$ gives rise to a new equivalence bimodule
\begin{equation}\label{eq:changeS}
(\Cinf(P_1)[[\lambda]], \pi^\1)
        \stackrel{\Psi^\1}{\longrightarrow}
        (\Cinf(S)[[\lambda]], \sigma)
        \stackrel{\Psi^\2}\longleftarrow
        (\Cinf(P_2)[[\lambda]], \pi^\2),
\end{equation}
with $\Psi^\ind = \exp(\mathcal{L}_Z)\Phi^\ind$, $i=1,2$, and $\sigma = \exp(\mathcal{L}_Z)\omega$. We will see now a less trivial way to modify equivalence bimodules using gauge transformations rather than maps.

In the classical geometric setting \cite{SeveraWeinstein2001}, given a Poisson structure $\pi$ on $P$ and a closed 2-form $B\in \Omega^2(P)$ such that $\id+ B^\flat\pi^\sharp\colon T^*P\to T^*P$ is invertible,  we obtain a new Poisson structure $\tau_B(\pi)$ on $P$ defined by the bundle map
$(\tau_{B}\pi)^\sharp:=\pi^\sharp(\id + B^\flat\pi^\sharp)^{-1}\colon
T^*P \to TP$. This operation is called {\em gauge transformation} of $\pi$ by $B$, while the closed 2-form $B$ is referred to as {\em B-field}. If we now have an equivalence bimodule
$$
(P_1, \pi_1)
        \stackrel{J_1}{\longleftarrow}
        (S, \omega)
        \stackrel{J_2}\longrightarrow
        (P_2, \pi_2)
$$
and a closed 2-form $B \in \Omega^2(P_1)$ with $\id+ B^\flat\pi_1^\sharp$ invertible, it is proven in \cite[Section~3]{BursztynRadko2003} that $\omega_B=\omega+ J_1^* B$ is symplectic and
$$
(P_1, \tau_B(\pi_1))
        \stackrel{J_1}{\longleftarrow}
        (S, \omega_B)
        \stackrel{J_2}\longrightarrow
        (P_2, \pi_2)
$$
is an equivalence bimodule. Our goal is to extend this result to the formal setting.

In the formal context, {\em $B$-fields} will be formal series $B = \sum_{j=0}^\infty\lambda^j B_j\in\Omega^2(P)[[\lambda]]$ of closed 2-forms.
For a formal Poisson structure
$\pi = \sum_{j=0}^\infty\lambda^j\pi_j$, we consider the map
$(\id+B^\flat\pi^\sharp)\colon\Omega^1(P)[[\lambda]] \to
\Omega^1(P)[[\lambda]]$. Notice that this map is invertible if and only if it is in zeroth order, i.e., if and only if $(\id +B_0^\flat\pi_0^\sharp)$ is invertible.
In this case, as shown in \cite{BDW-Characteristicclasses-starproducts}, we obtain a new formal Poisson structure $\tau_B(\pi)$ on $P$ via the property   \begin{equation}\label{eq:tauB}
    (\tau_B(\pi))^\sharp
    :=
    \pi^\sharp(\id+B^\flat\pi^\sharp)^{-1}\colon
    \Omega^1(P)[[\lambda]] \to \mathfrak{X}(P)[[\lambda]].
\end{equation}

We observe that Poisson maps behave well with respect to $B$-fields.

\begin{lemma}
    \label{lem:BPoiss}
Consider a Poisson morphism $\Phi = \exp(\mathcal{L}_Z)J^*\colon (C^\infty(P)[[\lambda]],\pi) \to (C^\infty(S)[[\lambda]],\omega)$, where $\pi$ is a formal Poisson structure and $\omega$ is a formal symplectic structure. Let $B=\sum_{j=0}^\infty \lambda^j B_j$ be a $B$-field on $P$ with
$\id +B_0^\flat\pi_0^\sharp$ invertible and $\omega_B = \omega+\exp(\mathcal{L}_{Z})(J^*B)$. Then $\omega_B$ is symplectic and
$\Phi\colon (C^\infty(P)[[\lambda]],\tau_B(\pi)) \to (C^\infty(S)[[\lambda]],\omega_B)$ is a Poisson morphism.
\end{lemma}

\begin{proof}
Composing $\Phi$ with the formal diffeomorphism $\exp(-\mathcal{L}_Z)$, we see that there is no loss in generality in assuming that  $\Phi=J^*$ for a smooth map $J\colon S\to P$. The fact that $\id +B_0^\flat\pi_0^\sharp$ is invertible guarantees that $\omega_0 + J^*B_0$ is nondegenerate (\cite[Section~3]{BursztynRadko2003}); since this is the zeroth order term of  $\omega_B = \omega+ J^*B$, it follows that $\omega_B$ is symplectic. So the corresponding map $\omega_B^\flat = \omega^\flat+ (J^*B)^\flat = (\id + (J^*B)^\flat(\omega^\flat)^{-1})\omega^\flat$ is invertible, which ensures that
$$
\id + (J^*B)^\flat(\omega^\flat)^{-1}\colon \Omega^1(S)[[\lambda]]\to \Omega^1(S)[[\lambda]]
$$
is invertible.

The fact that $\Phi=J^*\colon (C^\infty(P)[[\lambda]],\pi) \to (C^\infty(S)[[\lambda]],\omega)$ is a Poisson morphism can be phrased as the condition that, for any $\alpha \in \Omega^1(P)[[\lambda]]$, $J_*(\omega^\flat)^{-1}J^*\alpha = \pi^\sharp (\alpha)$ (here $J_*$  denotes ``$J$-relation'' of vector fields, naturally extended to the formal context). On the other hand, if $X$ is a formal vector field on $P$ such that $X=J_*Y$, for $Y\in \mathfrak{X}(S)[[\lambda]]$, then $(J^*B)^\flat (Y)=J^*B^\flat (X)$. It then follows that $(\id + (J^*B)^\flat(\omega^\flat)^{-1})J^* = J^*(\id+B^\flat\pi^\sharp)$, or, by taking inverses,
\begin{equation}\label{eq:J*rel}
(\id + (J^*B)^\flat(\omega^\flat)^{-1})^{-1}J^* = J^*(\id+B^\flat\pi^\sharp)^{-1}.
\end{equation}

Taking $\alpha = (\id + B^\flat\pi^\sharp)^{-1}\beta$ for $\beta\in \Omega^1(P)[[\lambda]]$, the condition that $\Phi=J^*\colon (C^\infty(P)[[\lambda]],\pi) \to (C^\infty(S)[[\lambda]],\omega)$ is a Poisson morphism says that
$$
J_*(\omega^\flat)^{-1}J^*(\id + B^\flat\pi^\sharp)^{-1}\beta = \pi^\sharp (\id + B^\flat\pi^\sharp)^{-1}\beta = (\tau_B(\pi))^\sharp(\beta).
$$
On the other hand, using \eqref{eq:J*rel}, we have that the left-hand side of this equation is
$$
J_*(\omega^\flat)^{-1}(\id + (J^*B)^\flat(\omega^\flat)^{-1})^{-1}J^*\beta
= J_*(\omega^\flat + (J^*B)^\flat)^{-1}J^*\beta = J_*(\omega_B^\flat)^{-1} J^*\beta.
$$
It follows that $J_*(\omega_B^\flat)^{-1} J^*\beta = (\tau_B(\phi))^\sharp(\beta)$ for all $\beta\in \Omega^1(P)[\lambda]]$, which is the condition for
$\Phi=J^*\colon (C^\infty(P)[[\lambda]],\tau_B(\pi)) \to (C^\infty(S)[[\lambda]],\omega_B)$ being a Poisson morphism.
\end{proof}

The following result provides many examples of Morita equivalent formal Poisson structures generated by $B$-fields.

\begin{theorem}
    \label{Bfieldactionondualpairs}%
    Consider an equivalence bimodule
    \begin{equation}
        \label{initialME}
        (\Cinf(P_1)[[\lambda]], \pi^\1)
        \stackrel{\Phi^\1}{\longrightarrow}
        (\Cinf(S)[[\lambda]], \omega)
        \stackrel{\Phi^\2}\longleftarrow
        (\Cinf(P_2)[[\lambda]], \pi^\2)
    \end{equation}
    and a $B$-field
    $B =\sum_{j=0}^\infty\lambda^jB_j$ on $P_1$
    such that  $(\id +B_0^\flat(\pi^\1_0)^\sharp)$ is invertible.
    Write $\Phi^\1=\exp(\mathcal{L}_{Z^\1})J_1^*$.
     Then
    \begin{equation}
        \label{newdualME}
        (\Cinf(P_1)[[\lambda]], \tau_B(\pi^\1))
        \stackrel{\Phi^\1}{\longrightarrow}
        (\Cinf(S)[[\lambda]], \omega_B)
        \stackrel{\Phi^\2}\longleftarrow
        (\Cinf(P_2)[[\lambda]], \pi^\2)
    \end{equation}
is an equivalence bimodule, where $\omega_B:=\omega+\exp(\mathcal{L}_{Z^\1})(J_1^*B)$.
\end{theorem}

For the proof, we start with a lemma.
For $h \in C^\infty(S)[[\lambda]]$, denote by $X_h$ and $X_h^B$ its hamiltonian vector fields relative to the symplectic forms $\omega$ and $\omega_B$, respectively.

\begin{lemma}
    \label{lem:hamVFB}
For all $f\in C^\infty(P_2)[[\lambda]]$, we have
$X_{\Phi^\2(f)}= X^B_{\Phi^\2(f)}.$
\end{lemma}

\begin{proof}
    Condition \eqref{altPoissoncommutation} of an equivalence bimodule says that, for any $f\in C^\infty(P_2)[[\lambda]]$,
$$
    \{ \Phi^\2(f), \Phi^\1(g)\}_\omega = i_{X_{\Phi^\2(f)}}d(\exp(\mathcal{L}_{Z^\1})J_1^*g)=
    i_{X_{\Phi^\2(f)}}(\exp(\mathcal{L}_{Z^\1})J_1^*dg)=0
$$
for all $g\in C^\infty(P_1)[[\lambda]]$, and the last equality implies that $i_{X_{\Phi^\2(f)}}(\exp(\mathcal{L}_{Z^\1})J_1^*\alpha)=0$ for any differential form $\alpha$ on $P_1$. In particular, $i_{X_{\Phi^\2(f)}}(\exp(\mathcal{L}_{Z^\1})J_1^*B)=0$. It follows that
$$
i_{X_{\Phi^\2(f)}}\omega_B = i_{X_{\Phi^\2(f)}}\omega = d\Phi^\2(f),
$$
and hence $X_{\Phi^\2(f)}= X^B_{\Phi^\2(f)}$.
\end{proof}

We now prove Theorem~\ref{Bfieldactionondualpairs}.

\begin{proof}
The zeroth order diagram corresponding to \eqref{newdualME} is
$$
(P_1, \tau_{B_0}(\pi^\1_0))
        \stackrel{J_1}{\longleftarrow}
        (S, \omega + J_1^*B_0)
        \stackrel{J_2}\longrightarrow
        (P_2, \pi^\2_0),
$$
i.e., the gauge transformation by $B_0$ of the zeroth order equivalence bimodule corresponding to \eqref{initialME}. The fact that this is again an equivalence bimodule is verified in \cite[Section~3]{BursztynRadko2003}.

To conclude that \eqref{newdualME} is an equivalence bimodule, we must now check that: (1) $\Phi^\1$ is a Poisson morphism, (2) $\Phi^\1(C^\infty(P_1)[[\lambda]])$ and $\Phi^\2(C^\infty(P_2)[[\lambda]])$ Poisson commute with respect to $\omega_B$, and (3) $\Phi^\2$ is an anti-Poisson morphism. Condition (1) follows directly from Lemma~\ref{lem:BPoiss}.
We note that (2) and (3) are direct consequences of Lemma~\ref{lem:hamVFB}, which shows that, for any $f\in C^\infty(P_2)[[\lambda]]$ and $h\in C^\infty(S)[[\lambda]]$, we have
$$
\{\Phi^\2(f),h\}_\omega = \mathcal{L}_{X_{\Phi^\2(f)}} h = \mathcal{L}_{X^B_{\Phi^\2(f)}} h = \{\Phi^\2(f),h\}_{\omega_B}.
$$
So for $g\in C^\infty(P_2)[[\lambda]]$ and $h=\Phi^\2(g)$,
$$
\Phi^\2(\{f,g\}^\2)= - \{\Phi^\2(f),\Phi^\2(g)\}_\omega = - \{\Phi^\2(f),\Phi^\2(g)\}_{\omega_B},
$$
which proves (3). Similarly, for  $g\in C^\infty(P_1)[[\lambda]]$,
$$
0=\{\Phi^\2(f),\Phi^\1(g)\}_\omega=
\{\Phi^\2(f),\Phi^\1(g)\}_{\omega_B},
$$
which proves (2).
\end{proof}


\section{Deformation of Poisson morphisms}
\label{sec:defmorphism}

As a first step in analyzing Problem~\ref{prob}, we consider the problem of deforming a Poisson morphism, focusing on its cohomological obstructions. Our discussion here will be purely algebraic.

Let $\alg{A}$ be a commutative algebra over a field $\mathbb{K} = \mathbb{R}$ or $\mathbb{C}$ (more generally, $\mathbb{K}$ could be a field of characteristic zero).
We denote by $\Der(\alg{A})$ the space of derivations of
 $\alg{A}$. We recall that $\Der(\alg{A}[[\lambda]]) = \Der(\alg{A})[[\lambda]]$, and  for $X\in \lambda\Der(\alg{A})[[\lambda]]$, $\exp(X)\colon \alg{A}[[\lambda]]\to \alg{A}[[\lambda]]$ is a ring automorphism (and any ring automorphism $\Phi=\sum_{k=0}^\infty \lambda^k \phi_k$ with $\phi_0=\id$ is of this form); in case $X$ is also a derivation of a Poisson bracket on $\alg{A}[[\lambda]]$, then  $\exp(X)$ is a Poisson automorphism. As in the geometric setting, two formal Poisson structures $\pi^\1$ and $\pi^\2$ on $\alg{A}[[\lambda]]$ are {\bf equivalent} if there exists  $X\in \lambda\Der(\alg{A})[[\lambda]]$ such that $\exp(X)\pi^\1=\pi^\2$.
Given a Poisson structure $\pi_0$ on $\alg{A}$, we denote by $\FPois_{\pi_0}(\alg{A})$ the set of equivalence classes of formal Poisson structures deforming $\pi_0$.

Let $\alg{A}$ and $\alg{B}$ be Poisson algebras over  $\mathbb{K}$ with Poisson brackets
$\pi_0 = \{\cdot, \cdot\}_{\alg{A}}$ and
$\sigma_0 = \{\cdot, \cdot\}_{\alg{B}}$, respectively. Throughout
this section we fix a Poisson morphism
\begin{equation}
    \label{eq:TheStartPoissonMorph}
    \phi_0\colon
    (\alg{A}, \{\cdot, \cdot\}_\alg{A})
    \rightarrow
    (\alg{B}, \{\cdot, \cdot\}_\alg{B}).
\end{equation}
We are interested in the problem of
deforming $\phi_0$ into a new Poisson morphism, once deformed Poisson
structures on $\alg{A}$ and $\alg{B}$ are fixed:
\begin{problem}
    \label{Deformationproblem}%
    Given formal Poisson deformations $\pi$ of $\pi_0$ and $\sigma$ of
    $\sigma_0$, find a derivation
    $X \in \lambda\Der(\alg{B})[[\lambda]]$ such that
    \begin{equation}
        \label{eq:FindPoissonMorph}
        \Phi
        :=
        \exp(X)\phi_0\colon
        (\alg{A}[[\lambda]], \pi)
        \rightarrow
        (\alg{B}[[\lambda]], \sigma)
    \end{equation}
    is a Poisson morphism.
\end{problem}

We observe that there are some natural degrees of freedom in solving the previous problem. We say that a derivation $V \in \Der(\alg{B})[[\lambda]]$ is {\bf vertical} if  $V \circ \phi_0 = 0$; note that the vertical derivations form a Lie
subalgebra of all derivations. Given a solution $X$ to Problem~\ref{Deformationproblem}, it is clear that the derivation $\overline{X}$
determined by
 \begin{equation}
        \label{eq:EquivalenceXccX}
        \exp(\overline{X}) = \exp(Y)\exp(X)\exp(V),
    \end{equation}
where $Y, V \in \lambda\Der(\alg{B})[[\lambda]]$, $V$ is vertical and $Y$ is a Poisson derivation of $\sigma$, is a new solution to the problem. In this case we say that the solutions $X$ and $\overline{X}$ are {\bf equivalent}.

We now discuss  the existence
and uniqueness of solutions to Problem~\ref{Deformationproblem}.

%
%

\subsection{The cohomology controlling the problem}
\label{subsec:CEcoh}

The undeformed Poisson morphism $\phi_0$ gives rise to a Lie algebra
morphism
\begin{equation}
    \label{eq:AtoEndosOfB}
    (\alg{A}, \{\cdot,\cdot\}_\alg{A})
    \rightarrow
    (\End_{\mathbb{K}}(\alg{B}), [\cdot,\cdot]),
    \quad
    \;\;\;
    \quad
    a \mapsto \{\phi_0(a), \cdot\}_\alg{B}.
\end{equation}
So we can see $\alg{B}$ as a left Lie-algebra module over
$(\alg{A}, \{\cdot, \cdot\}_\alg{A})$. We consider the
following standard Chevalley-Eilenberg complex: for $k = 0$ set
$\CCE^0(\alg{A}, \alg{B}) := \alg{B}$, and for
$k \in \mathbb{N}$, the $k$-cochains are
\begin{equation}
    \label{eq:CCEDef}
    \CCE^k(\alg{A}, \alg{B})
    :=
    \left\{
        D\colon\alg{A}^{\times k} \rightarrow \alg{B}
        \; \Big| \;
        D \textrm{ is multilinear and antisymmetric}
    \right\}.
\end{equation}
The Chevalley-Eilenberg differential $\delta$ is given, for
$D\in \CCE^k(\alg{A}, \alg{B})$ and
$a_0, \ldots, a_{k} \in \alg{A}$, by
\begin{align}\label{eq:CEdiff}
    &(\delta D)(a_0, \ldots, a_{k}) \\ \nonumber
    &\quad:=
    \sum_{j=0}^{k} (-1)^{j}
    \big\{
    \phi_0(a_j),
    D(a_0, \ldots, \widehat{a_j}, \ldots, a_{k})
    \big\}_\alg{B}
    +
    \sum_{i<j}(-1)^{i+j}
    D\big(
    \{a_i, a_j\}_\alg{A},
    a_0, \ldots,
    \widehat{a_i}, \ldots,
    \widehat{a_j}, \ldots, a_{k}
    \big).
\end{align}
The corresponding cohomology will be denoted by
$\HCE^\bullet(\alg{A}, \alg{B})$.
\begin{remark}
    \label{remark:CohomologyAlwaysForUndeformedPoisson}%
    We stress that the
    Chevalley-Eilenberg cohomologies that we consider throughout are with respect to the {\em undeformed} Poisson
    structures $\pi_0$ and $\sigma_0$ on $\alg{A}$ and $\alg{B}$,
    respectively.
\end{remark}
\begin{remark}
    \label{Modulestructurefortherelevantcomplex}%
    The spaces $\CCE^k(\alg{A}, \alg{B})$ can be given the structure
    of a (say, left) $\alg{B}$-module via
    \begin{equation}
        \label{eq:CCEBleftModule}
        (bD)(a_1, \ldots, a_k)
        :=
        bD(a_1,\ldots, a_k),
    \end{equation}
    for $b \in \alg{B}$, $D \in \CCE^k(\alg{A}, \alg{B})$,
    and $a_i\in\alg{A}$. Using $\phi_0,$ they also have a
    left $\alg{A}$-module structure via
    \begin{equation}
        \label{eq:CCEAleftModule}
        (aD)(a_1, \ldots, a_k)
        :=
        \phi_0(a)D(a_1, \ldots, a_k),
    \end{equation}
    where $a \in \alg{A}$. In general, the
    differential is not compatible with these module
    structures.
\end{remark}

The Chevalley-Eilenberg complex plays a central role in the study of
deformations of Lie brackets and Lie-algebra morphisms \cite{Nijenhuis1967-DefLieAlgebras, Nijenhuis1967-DefMorphismsLieGroups}. In order to handle Poisson algebras,
we need to consider a subcomplex taking into account the Leibniz rule of
Poisson brackets.  So we focus on
the special class of cochains given by multiderivations
of the associative product along $\phi_0$, i.e., for
$D \in \CCE^k(\alg{A}, \alg{B})$, we require the additional condition
\begin{equation}
    \label{dercochains}
    D(a_1, \ldots, a_la'_l, \ldots, a_k)
    =
    \phi_0(a_l)D(a_1, \ldots,a'_l, \ldots, a_k)
    +
    D(a_1, \ldots, a_l, \ldots, a_k)\phi_0(a'_l),
\end{equation}
for $a_1, \ldots, a_k, a'_l \in \alg{A}$ and $l = 1, \ldots, k$.  Let
$\CCEder^k(\alg{A}, \alg{B})$ denote the subset of
$\CCE^k(\alg{A}, \alg{B})$ satisfying \eqref{dercochains}. The differential $\delta$
restricts to this subset, so we have a subcomplex
$(\CCEder^k(\alg{A}, \alg{B}), \delta)$. The corresponding cohomology
will be denoted by $\HCEder^\bullet(\alg{A}, \alg{B})$. Note that the
inclusion at the level of cochains induces a map
$\HCEder^\bullet(\alg{A}, \alg{B})\rightarrow \HCE^\bullet(\alg{A},
\alg{B})$.

Given another Poisson morphism
$\psi\colon\alg{B}\to\alg{C}$, we have a map
$\phi_0^*\colon \CCE^\bullet(\alg{B}, \alg{C})\rightarrow
\CCE^\bullet(\alg{A}, \alg{C})$,
\begin{equation}
    \label{eq:GeneralPullBack}
    (\phi_0^*D)(a_1, \ldots, a_k)
    =
    D(\phi_0(a_1), \ldots, \phi_0(a_k)).
\end{equation}
A direct computation shows that this map is a chain
map. Moreover, taking into account that $\phi_0$ is a morphism of
commutative products, we see that the above map restricts to a
chain map
\begin{equation}
        \label{eq:ChainMapCCEderBCStuff}
        \phi_0^*\colon
        \CCEder^\bullet(\alg{B}, \alg{C})
        \rightarrow
        \CCEder^\bullet(\alg{A}, \alg{C}).
\end{equation}

Taking $\alg{C} = \alg{B}$ and $\psi = \id_\alg{B}$, we get a chain map
$\phi_0^*\colon \CCEder^\bullet(\alg{B}, \alg{B})
\rightarrow\CCEder^\bullet(\alg{A}, \alg{B})$. Notice that $\CCEder^1(\alg{B}, \alg{B}) = \Der(\alg{B})$.



We will need the concept of {\em horizontal lift} in this purely algebraic setting:

\begin{definition}[\textbf{Horizontal lift}]
    \label{horizlift}%
    A horizontal lift along
    $\phi_0\colon\alg{A}\rightarrow\alg{B}$ is an
    $\alg{A}$-linear map
    \begin{equation}
        h\colon \CCE^1(\alg{A}, \alg{B})\to
                \CCE^1(\alg{B}, \alg{B})=\Der(\alg{B}),\qquad D
        \; \mapsto \; D^h,
    \end{equation}
    such that for all
    $D \in \CCE^1(\alg{A}, \alg{B})$,
    \begin{equation}
        \label{horproperty}
        \phi_0^*D^h = D.
    \end{equation}
\end{definition}
Here the $\alg{A}$-linearity is also understood along
$\phi_0$, i.e.,
\begin{equation}
    \label{eq:HorLiftALinear}
    (aD)^h = \phi_0(a) D^h
\end{equation}
for all $a \in \alg{A}$ and
$D \in \CCE^1(\alg{A}, \alg{B})$.
It follows from the definition that two horizontal lifts $h_1$ and
$h_2$ differ by vertical derivation.

In the geometric context of interest, these horizontal lifts will be defined by usual horizontal
lifts along surjective submersions given by a connection (see Section~\ref{sec:ClassifyingMap}).

%
%

\subsection{Existence of deformations}
\label{subsec:DeformationExistence}

We now discuss
the existence of a derivation $X$ solving Problem~ \ref{Deformationproblem}; we will do that
assuming that we have a horizontal lift along
$\phi_0$ as in Definition~\ref{horizlift}.

We will look for a solution $\Phi = \exp(X)\phi_0$ of the problem in the usual inductive way. Suppose that we have found $X_{(k)} \in \lambda \Der(\alg{B})[[\lambda]]$
such that $\Phi_{(k)} := \exp(X_{(k)})$ solves Problem
\ref{Deformationproblem} up to order $k$, i.e., we have
\begin{equation}
    \label{Poissonmorphismconduptok}
    \Phi_{(k)}^{-1}\sigma
    \big(\Phi_{(k)}\phi_0(a_1), \Phi_{(k)}\phi_0(a_2)\big)
    =
    \phi_0(\pi(a_1, a_2))
    +
    \lambda^{k+1}R_{k+1}(a_1, a_2)
    + \cdots
\end{equation}
for all $a_1, a_2 \in \alg{A}$, and some bilinear map
$R_{k+1}\colon \alg{A} \times \alg{A} \rightarrow \alg{B}$.

\begin{lemma}
    \label{PropertiesofRkplusOne}%
    The map $R_{k+1}$ defined by
    Equation~\eqref{Poissonmorphismconduptok} is a $2$-cocycle in
    $\CCEder^\bullet(\alg{A}, \alg{B})$.
\end{lemma}
\begin{proof}
    Since $\pi$ and $\sigma$ are both antisymmetric in each order of
    $\lambda$, this also holds for $R_{k+1}$. Moreover, $\sigma$ and
    $\pi$ satisfy the Leibniz rule and $\Phi_{(k)}$ is an automorphism
    of the associative algebra $\alg{B}[[\lambda]]$, thus the left
    hand side in \eqref{Poissonmorphismconduptok} also satisfies the
    Leibniz rule. So the same holds for $R_{k+1}$. As for
    closedness, we compute, for $a_1, a_2, a_3 \in \alg{A}$:
    \begin{align*}
        &\Phi_{(k)}^{-1}\sigma\Big(
        \Phi_{(k)}\phi_0(a_1),
        \sigma\big(
        \Phi_{(k)}\phi_0(a_2),
        \Phi_{(k)}\phi_0(a_3)
        \big)
        \Big) \\
        &\quad=
        \Phi_{(k)}^{-1}\sigma\Big(
        \Phi_{(k)}\phi_0(a_1),
        \Phi_{(k)}\big(
        \phi_0(\pi(a_2, a_3))
        + \lambda^{k+1}R_{k+1}(a_2,a_3)
        + \cdots
        \big)
        \Big) \\
        &\quad=
        \Phi_{(k)}^{-1}\sigma\Big(
        \Phi_{(k)}\phi_0(a_1),
        \Phi_{(k)}\phi_0 (\pi(a_2, a_3))
        \Big) \\
        &\quad\quad+
        \Phi_{(k)}^{-1}\sigma\Big(
        \Phi_{(k)}\phi_0(a_1),
        \Phi_{(k)}\big(
        \lambda^{k+1}R_{k+1}(a_2, a_3)
        +
        \cdots
        \big)
        \Big) \\
        &\quad=
        \phi_0\big(\pi(a_1, \pi(a_2, a_3))\big)
        +
        \lambda^{k+1}R_{k+1}(a_1, \pi(a_2, a_3))
        +
        \cdots\\
        &\quad\quad+
        \Phi_{(k)}^{-1}\sigma\Big(
        \Phi_{(k)}\phi_0(a_1),
        \Phi_{(k)}\big(
        \lambda^{k+1}R_{k+1}(a_2, a_3)
        +
        \cdots
        \big)
        \Big),
    \end{align*}
    where we used the defining
    Equation~\eqref{Poissonmorphismconduptok} of $R_{k+1}$ for the
    first and last equalities. Taking the cyclic sum over $a_1, a_2, a_3$,
    using the Jacobi identity for $\pi$ and for $\sigma$ (recall that
    $\Phi_{(k)}$ is an automorphism, so it turns $\sigma$ into a
    formal Poisson bracket again), and the fact that $\Phi_{(k)}$ is
    the identity at zeroth order, we obtain
    \begin{align*}
        0
        &=
        \lambda^{k+1}\Big(
        R_{k+1}\big(a_1, \pi(a_2, a_3)\big)
        +
        R_{k+1}\big(a_2, \pi(a_3, a_1)\big)
        +
        R_{k+1}\big(a_3, \pi(a_1, a_2)\big)
        \Big) \\
        &\quad+
        \lambda^{k+1}\Big(
        \sigma\big(\phi_0(a_1), R_{k+1}(a_2, a_3)\big)
        +
        \sigma\big(\phi_0(a_2), R_{k+1}(a_3, a_1)\big)
        +
        \sigma\big(\phi_0(a_3), R_{k+1}(a_1, a_2)\big)
        \Big)
        +
        \cdots.
    \end{align*}
    To obtain the term of order $\lambda^{k+1}$ from this equation, we
    have to expand the deformed Poisson structures $\pi$ and $\sigma$
    and take their zeroth order terms, which are the original brackets on
    $\alg{A}$ and $\alg{B}$. Thus we get
    \begin{align*}
        0
        &=
        R_{k+1}(a_1, \{a_2, a_3\}_\alg{A})
        +
        R_{k+1}(a_2, \{a_3, a_1\}_\alg{A})
        +
        R_{k+1}(a_3, \{a_1, a_2\}_\alg{A}) \\
        &\quad+
        \{\phi_0(a_1), R_{k+1}(a_2,a_3)\}_\alg{B}
        +
        \{\phi_0(a_2), R_{k+1}(a_3,a_1)\}_\alg{B}
        +
        \{\phi_0(a_3), R_{k+1}(a_1,a_2)\}_\alg{B},
    \end{align*}
    which is precisely the condition $\delta R_{k+1} = 0$.
\end{proof}

\begin{proposition}[\textbf{Existence}]
    \label{Existence}%
    Suppose that we have a horizontal lift.  If
    $\HCEder^2(\alg{A}, \alg{B}) = \{0\}$, then there exists a
    derivation $X \in \lambda\Der(\alg{B})[[\lambda]]$ such that
    $\Phi = \exp(X)\phi_0$ solves Problem~\ref{Deformationproblem}.
\end{proposition}
\begin{proof}
    We construct a solution inductively. First notice that
    $X_{(0)} = 0$ does the job for $k = 0$. In this case
    $\Phi_{(0)} = \id_\mathcal{B}$ and we have
    \begin{equation*}
        R_{1}(a_1, a_2)
        =
        \sigma_1\big(\phi_0(a_1), \phi_0(a_2)\big)
        -
        \phi_0(\pi_1(a_1, a_2)).
    \end{equation*}
    Now suppose we have found $X_{(k)}$ such that Equation
    \eqref{Poissonmorphismconduptok} holds. We look for
    $X_{k+1}\in \Der(\mathcal{B})$ such that, for
    $\Phi_{(k+1)}:= \Phi_{(k)}\exp(\lambda^{k+1}X_{k+1})$, we have an
    analogous equation to \eqref{Poissonmorphismconduptok} up to one
    order higher. Then
    \begin{align*}
&        \Phi_{(k+1)}^{-1}\sigma
        \big(\Phi_{(k+1)}\phi_0(a_1), \Phi_{(k+1)}\phi_0(a_2)\big)
         =
        \Phi_{k+1}^{-1}\Phi_{(k)}^{-1}\sigma\big(
        \Phi_{(k)}\Phi_{k+1}\phi_0(a_1),
        \Phi_{(k)}\Phi_{k+1}\phi_0(a_2)
        \big) \\
        & =
        (\id - \lambda^{k+1}X_{k+1} + \cdots)
        \Phi_{(k)}^{-1}\sigma\big(
        \Phi_{(k)} \big(
        (\id + \lambda^{k+1}X_{k+1} + \cdots)
        \phi_0(a_1)
        \big),
        a_1 \leftarrow a_2 \big),
    \end{align*}
    where $a_1 \leftarrow a_2$ in the second argument of $\sigma$ means
    that we repeat the first argument with $a_1$ replaced by
    $a_2$. Expanding the right-hand side and using the defining equation
    of $R_{k+1}$, the fact that $\Phi_{(k)}$ and $\Phi_{(k)}^{-1}$ are
    the identity in zeroth order, and that $\phi_0$ is a Poisson
    morphism between $\{\cdot, \cdot\}_\alg{A}$ and
    $\{\cdot, \cdot\}_\alg{B}$, we obtain
    \begin{align*}
        \Phi_{(k+1)}^{-1}\sigma\big(
        \Phi_{(k+1)}\phi_0(a_1),
        \Phi_{(k+1)}\phi_0(a_2)
        \big)
        =&
        \Phi_{(k)}^{-1}\sigma\big(
        \Phi_{(k)}\phi_0(a_1),
        \Phi_{(k)}\phi_0(a_2)
        \big)
        -
        \lambda^{k+1}X_{k+1}\sigma
        \big(\phi_0(a_1), \phi_0(a_2)\big) \\
        &+
        \lambda^{k+1}\sigma
        \big(\phi_0(a_1), X_{k+1}\phi_0(a_2)\big)
        +
        \lambda^{k+1}\sigma
        \big(X_{k+1}\phi_0(a_1), \phi_0(a_2)\big)
        +
        \cdots \\
        =&
        \phi_0(\pi(a_1, a_2))
        +
        \lambda^{k+1}R_{k+1}(a_1, a_2)
        -
        \lambda^{k+1}X_{k+1}\phi_0\{a_1, a_2\}_\alg{A} \\
        &+
        \lambda^{k+1}
        \{\phi_0(a_1), X_{k+1}\phi_0(a_2)\}_\alg{B}
        +
        \lambda^{k+1}
        \{X_{k+1}\phi_0(a_1), \phi_0(a_2)\}_\alg{B}
        +
        \cdots.
    \end{align*}
    Hence to solve Problem~\ref{Deformationproblem} up to order $k+1$
    we need to find $X_{k+1}$ satisfying
    \begin{equation*}
        R_{k+1}(a_1, a_2)
        =
        X_{k+1}\phi_0\{a_1, a_2\}_\alg{A}
        -
        \{\phi_0(a_1), X_{k+1}\phi_0(a_2)\}_\alg{B}
        -
        \{X_{k+1}\phi_0(a_1), \phi_0(a_2)\}_\alg{B},
    \end{equation*}
    for all $a_1, a_2 \in \alg{A}$, which is just the equation
    \begin{equation}
        \label{CondforextensionuptokplusOne}
        - R_{k+1} = \delta \phi_0^*X_{k+1}.
    \end{equation}
    Since $R_{k+1}$ is closed in
    $\CCEder^2(\alg{A}, \alg{B})$, the condition
    $\HCEder^2(\alg{A}, \alg{B}) = \{ 0 \}$ implies that there exists
    $\xi_{k+1} \in \CCEder^1(\alg{A}, \alg{B})$ with
    $-R_{k+1} = \delta \xi_{k+1}$. Then any horizontal lift
    $X_{k+1} := \xi_{k+1}^h$ of $\xi_{k+1}$ satisfies
    \eqref{CondforextensionuptokplusOne}.
\end{proof}

%
%

\subsection{Uniqueness of deformations}
\label{subsec:uniqueness}
In this section we provide conditions for the uniqueness of solutions of Problem~\ref{Deformationproblem}, up to the natural notion of equivalence \eqref{eq:EquivalenceXccX}.

For a Poisson structure $\sigma$ on the ring $\alg{B}[[\lambda]]$, any element $H = H_0 + \lambda H_1 + \cdots \in \alg{B}[[\lambda]]$ defines an {\em inner} Poisson derivation $X_H$ of $(\alg{B}[[\lambda]],\sigma)$ by
\begin{equation}
    \label{eq:HamiltonianDerivation}
    X_H b = \sigma(H, b).
\end{equation}
We refer to such inner derivations as {\bf Hamiltonian derivations}, noticing that they form a Lie ideal inside the Poisson derivations. In particular,  the Baker-Campbell-Hausdorff series of two Hamiltonian derivations is again Hamiltonian, a fact we will frequently use.

We can now state our main uniqueness result.
\begin{proposition}[\textbf{Uniqueness}]
    \label{Uniqueness}%
    Assume that $\HCEder^1(\alg{A}, \alg{B}) = \{0\}$. If
    $X, \overline{X} \in \lambda \Der(\alg{B})[[\lambda]]$ are two solutions to
    Problem~\ref{Deformationproblem}, then
    \begin{equation}
        \label{eq:TheFunnyTrafo}
        \exp(\overline{X}) = \exp(X_H)\exp(X)\exp(V)
    \end{equation}
    for $V \in \lambda\Der(\alg{B})[[\lambda]]$ a vertical derivation and
    $X_H = \sigma(H, \cdot)$ with $H \in \lambda\alg{B}[[\lambda]]$.
    In particular, any two solutions to Problem~\ref{Deformationproblem} are equivalent.
\end{proposition}

\begin{proof}
    Suppose that we have found
    $Y_{(k)}, V_{(k)} \in \lambda\Der(\alg{B})[[\lambda]]$ such that
    $V_{(k)}$ is vertical, $Y_{(k)}$ is a derivation of $\sigma$ and
    $\exp(Y_{(k)}) \exp(X) \exp(V_{(k)})$ agrees with
    $\exp(\overline{X})$ up to order $k$. Notice that
    $Y_{(0)} = V_{(0)} = 0$ solves the case $k = 0$, which is the
    starting point for our recursive construction.

    The fact that $\exp(Y_{(k)})\exp(X)\exp(V_{(k)})$ agrees with
    $\exp(\overline{X})$ up to order $k$ means that there is a
    $Z_{(k+1)} \in \lambda^{k+1} \Der(\alg{B})[[\lambda]]$ such that
    \begin{equation}
        \label{uniquenesuptok}
        \exp(Y_{(k)})\exp(X)\exp(V_{(k)})\exp(-\overline{X})
        =
        \exp(Z_{(k+1)})
        =
        \id + \lambda^{k+1}Z_{k+1} + \cdots.
    \end{equation}
    Now we look for a vertical derivation
    $V_{k+1} \in \Der(\alg{B})[[\lambda]]$ and a Poisson derivation
    $Y_{k+1} \in \Der(\alg{B})[[\lambda]]$ with respect to the
    deformed Poisson structure $\sigma$ such that the maps
    $\exp(\overline{X})$ and
    $$
    \exp(\lambda^{k+1}Y_{k+1}) \exp(Y_{(k)}) \exp(X) \exp(V_{(k)})
    \exp(\lambda^{k+1}V_{k+1})
    $$
    agree up to order $k+1$.  Computing the analogue of
    Equation~\eqref{uniquenesuptok} up to one order higher gives
    \begin{align*}
        &\exp(\lambda^{k+1}Y_{k+1})
        \exp(Y_{(k)})
        \exp(X)
        \exp(V_{(k)})
        \exp(\lambda^{k+1}V_{k+1})
        \exp(-\overline{X}) \\
        &\quad=
        (\id + \lambda^{k+1}Y_{k+1}^0 + \cdots)
        \exp(Y_{(k)}) \exp(X) \exp(V_{(k)})
        (\id + \lambda^{k+1}V_{k+1}^0 + \cdots)
        \exp(-\overline{X}) \\
        &\quad=
        \exp(Y_{(k)}) \exp(X) \exp(V_{(k)}) \exp(-\overline{X})
        +
        \lambda^{k+1}(Y_{k+1}^0 + V_{k+1}^0)
        +
        \cdots \\
        &\quad=
        \id
        + \lambda^{k+1}Z_{k+1}
        + \lambda^{k+1}(Y_{k+1}^0
        + V_{k+1}^0)
        + \cdots.
    \end{align*}
    So the zeroth order terms $Y_{k+1}^0$ and $V_{k+1}^0$ need to
    satisfy
    \begin{equation}\label{CondonYandVinkPlusOne}
        Y_{k+1}^0 + V_{k+1}^0 = -Z_{k+1}.
    \end{equation}
    Since $\exp(Y_{(k)})$ is an automorphism of $\sigma$ we know
    that $\exp(Y_{(k)}) \exp(X) \exp(V_{(k)}) \phi_0$ solves
    Problem~\ref{Deformationproblem}. Thus we compute
    \begin{align}
        &\exp(Z_{(k+1)})
        \sigma\Big(
        \exp(-Z_{(k+1)}) \exp(Y_{(k)}) \exp(X) \exp(V_{(k)}) \phi_0(a_1),
        a_1\leftarrow a_2
        \Big)
        \notag \\
        &\quad=
        \exp(Y_{(k)}) \exp(X) \exp(V_{(k)}) \exp(-\overline{X})
        \sigma\Big(
        \exp(\overline{X})\phi_0(a_1),
        \exp(\overline{X})\phi_0(a_2)
        \Big)
        \notag \\
        &\quad=
        \exp(Y_{(k)}) \exp(X) \exp(V_{(k)}) \phi_0(\pi(a_1, a_2))
        \notag \\
        \label{computationOne}
        &\quad=
        \sigma\Big(
        \exp(Y_{(k)})\exp(X)\exp(V_{(k)})\phi_0(a_1),
        a_1\leftarrow a_2
        \Big).
    \end{align}
    Expanding the left hand side of this equation we get
    \begin{align*}
        &\exp(Z_{(k+1)})
        \sigma\Big(
        \exp(-Z_{(k+1)}) \exp(Y_{(k)}) \exp(X) \exp(V_{(k)})
        \phi_0(a_1),
        a_1\leftarrow a_2
        \Big) \\
        &\quad=
        (\id + \lambda^{k+1}Z_{k+1} + \cdots)
        \sigma\Big(
        (\id - \lambda^{k+1}Z_{k+1} + \cdots)
        \exp(Y_{(k)})\exp(X) \exp(V_{(k)})\phi_0(a_1),
        a_1\leftarrow a_2
        \Big) \\
        &\quad=
        \sigma\Big(
        \exp(Y_{(k)}) \exp(X) \exp(V_{(k)}) \phi_0(a_1),
        \exp(Y_{(k)}) \exp(X) \exp(V_{(k)}) \phi_0(a_2)
        \Big) \\
        &\quad\quad+
        \lambda^{k+1}\Big(
        Z_{k+1}\phi_0\{a_1, a_2\}_\alg{A}
        -
        \{Z_{k+1}\phi_0(a_1), \phi_0(a_2)\}_\alg{B}
        -
        \{\phi_0(a_1), Z_{k+1}\phi_0(a_2)\}_\alg{B}
        \Big)
        + \cdots \\
        &\quad=
        \sigma\Big(
        \exp(Y_{(k)}) \exp(X) \exp(V_{(k)}) \phi_0(a_1),
        a_1\leftarrow a_2
        \Big)
        -
        \lambda^{k+1}(\delta\phi_0^*Z_{k+1})(a_1, a_2)
        +
        \cdots.
    \end{align*}
    Hence, comparing with the right-hand side of
    \eqref{computationOne}, we see that
    \begin{equation*}
        \delta\phi_0^*Z_{k+1} = 0.
    \end{equation*}
    So the condition $\HCEder^1(\alg{A}, \alg{B}) = \{0\}$ implies
    that we can find $H_{k+1} \in \alg{B}$ such that
    $\delta H_{k+1} = \phi_0^*Z_{k+1}$, i.e.,
    \begin{equation*}
        \{H_{k+1}, \phi_0(a)\}_\alg{B}
        =
        -\phi_0^*Z_{k+1}(a) = -Z_{k+1}(\phi_0(a)).
    \end{equation*}
    Observe that
    $H_{k+1}\in \CCEder^0(\alg{A}, \alg{B}) = \alg{B} =
    \CCEder^0(\alg{B}, \alg{B})$.
    It follows from the previous equation that
    $V_{k+1}^0:= \delta H_{k+1} - Z_{k+1}\in \CCEder^1(\alg{B},
    \alg{B})$
    vanishes along $\phi_0$, so it is a vertical derivation. Then we
    have
    \begin{equation*}
        \{H_{k+1}, \cdot\}_\alg{B} + V_{k+1}^0 = -Z_{k+1},
    \end{equation*}
    so we fulfill \eqref{CondonYandVinkPlusOne} by taking
    $Y_{k+1}^0 := \{H_{k+1}, \cdot\}_\alg{B}$. Now, we put
    $V_{k+1} := V_{k+1}^0$ and $Y_{k+1} := \sigma(H_{k+1}, \cdot)$,
    which is a Hamiltonian (hence Poisson) derivation of $\sigma$ agreeing with $Y_{k+1}^0$ in zeroth order of $\lambda$.
\end{proof}

%
%

\subsection{Commutants}\label{subsec:commutants}
As a second step to analyze Problem~\ref{prob}, we return to
the Poisson morphism $\phi_0\colon (\alg{A},\{\cdot,\cdot\}_{\alg{A}})\to (\alg{B},\{\cdot,\cdot\}_{\alg{B}})$ as starting point, and consider the commutant of $\phi_0(\alg{A})$ inside $\alg{B}$,
\begin{equation}
    \label{eq:DefinitionCommutant}
    \alg{A}'
    :=
    (\phi_0(\alg{A}))^c
    :=
    \left\{
        b \in \alg{B}
        \; \big| \;
        \{b, \phi_0(a)\}_\alg{B} = 0
        \textrm{ for all }
        a \in \alg{A}
    \right\}.
\end{equation}
We will keep using the notation $\pi_0=\{\cdot,\cdot\}_{\alg{A}}$, $\sigma_0=\{\cdot,\cdot\}_{\alg{B}}$.
Since $\alg{A}'$ is a Poisson subalgebra of $\alg{B}$, it acquires a Poisson structure $\pi_0'$. Denoting by $\phi_0'\colon \alg{A}'\to \alg{B}$ the inclusion map, we obtain a diagram of Poisson morphisms
\begin{equation}
    \label{eq:poissonmaps}
    (\alg{A},\pi_0)\stackrel{\phi_0}{\longrightarrow} (\alg{B},\sigma_0) \stackrel{\phi_0'}{\longleftarrow}(\alg{A}',\pi'_0)
\end{equation}
with Poisson commuting images. We are interested in studying deformations of this diagram. Having studied the deformation problem for $\phi_0$ in Problem~\ref{Deformationproblem}, we now consider the right leg of the diagram.

The set-up in this section is  that we have formal Poisson deformations $\pi$ of $\pi_0$ and $\sigma$ of $\sigma_0$, and a Poisson morphism $\Phi\colon (\alg{A}[[\lambda]],\pi)\to (\alg{B}[[\lambda]],\sigma)$ deforming $\phi_0$.   We will be concerned with the following two issues.
\begin{itemize}
    \item[(A)] Find a morphism of rings $\Phi'\colon \alg{A}'[[\lambda]]\to \alg{B}[[\lambda]]$ deforming $\phi_0'$ such that
    $$
    \sigma(\Phi'(a'),\Phi(a))=0
    $$
    for all $a\in \alg{A}[[\lambda]]$ and $a'\in \alg{A}'[[\lambda]]$.
    \item[(B)] Find a formal Poisson deformation $\pi'$ of $\pi'_0$ for which $\Phi'$ is a Poisson morphism into $(\alg{B}[[\lambda]],\sigma)$
\end{itemize}
With such $\Phi'$ and $\pi'$ we obtain a diagram of Poisson morphisms
\begin{equation}
    \label{eq:deformedpoissonmaps}
    (\alg{A}[[\lambda]],\pi)\stackrel{\Phi}{\longrightarrow} (\alg{B}[[\lambda]],\sigma) \stackrel{\Phi'}{\longleftarrow}(\alg{A}'[[\lambda]],\pi')
\end{equation}
with Poisson commuting images deforming \eqref{eq:poissonmaps}.
We will see below that solving (A) automatically solves (B), so we start focusing on (A).

Let
\begin{equation}
    \label{eq:NewDeformedCommutant}
    C :=
    \left\{
        b \in \alg{B}[[\lambda]]
        \; \big| \;
        \sigma(b, \Phi(a)) = 0
        \textrm{ for all } a \in \alg{A}[[\lambda]]
    \right\}
\end{equation}
be the commutant of the image $\Phi(\alg{A}[[\lambda]])$ with
respect to $\sigma$. We conveniently reformulate (A) as follows.

\begin{problem}\label{prob:comm}
Find a derivation $X' \in \lambda \Der(\alg{B})[[\lambda]]$ such that
$\Phi'= \exp(X')\phi_0'\colon \alg{A}'[[\lambda]] \to \alg{B}[[\lambda]]$ satisfies $\Phi'(\alg{A}'[[\lambda]])\subseteq C$, i.e.,
\begin{equation}\label{eq:poiscomm}
    \sigma\big(\exp(X')\phi_0'(a'), \Phi(a)\big)
    = 0
\end{equation}
for all $a\in \alg{A}$, $a' \in \alg{A}'$.
\end{problem}

Solving Problem~\ref{prob:comm} has the following consequences (which in particular solves (B)).

\begin{proposition}
    \label{Deformationofthecommutant}%
Let  $X' \in \lambda \Der(\alg{B})[[\lambda]]$ be such that $\exp(X')\phi_0'(\alg{A}'[[\lambda]])\subseteq C$. Then
\begin{itemize}
\item[(i)] the map
    \begin{equation}
        \label{eq:IsoClassicalDeformedCommutant}
        \exp(X')\phi_0'\colon\alg{A}'[[\lambda]] \rightarrow C
    \end{equation}
    is a ring isomorphism;
    \item[(ii)] there is a unique Poisson structure $\pi'$
    on $\alg{A}'[[\lambda]]$ such that $\Phi'= \exp(X')\phi_0'\colon (\alg{A}'[[\lambda]],\pi')\to (\alg{B}[[\lambda]],\sigma)$ is a Poisson morphism;
            \end{itemize}
\end{proposition}

\begin{proof}
    Part (i) follows from the exact same argument as in the proof of Prop.~\ref{prop:commimage}.

     Since $C$ is a Poisson subalgebra of  $(\alg{B}[[\lambda]],\sigma)$, the isomorphism \eqref{eq:IsoClassicalDeformedCommutant} in (i) defines a Poisson structure on $\alg{A}'[[\lambda]]$ as desired. 
        \end{proof}

We will say that two solutions $X'$ and $X''$ to Problem~\ref{prob:comm} are {\bf equivalent} if there is a derivation
    $\xi \in \lambda \Der(\alg{A'})[[\lambda]]$ such that
  \begin{equation}\label{eq:equivprobcomm}
      \exp(X')\phi_0' = (\exp(X'')\phi_0')\circ \exp(\xi).
  \end{equation}

With this notion, we have the following uniqueness results.
\begin{proposition}\label{prop:unique}
The following holds.
\begin{itemize}
\item[(i)] Any two solutions  to Problem~\ref{prob:comm} are equivalent.
\item[(ii)] The equivalence class of $\pi'$ in Prop.~\ref{Deformationofthecommutant} (ii) is independent of the choice of derivation $X'$.
\end{itemize}
\end{proposition}

\begin{proof}
    Let $X'$ and $X''$ be two solutions to Problem~\ref{prob:comm}. Given $a' \in \alg{A}'[[\lambda]]$, by part (i) of the previous proposition there is a unique $a'' \in \alg{A}'[[\lambda]]$ such that
$$
        \exp(X'') \phi_0'(a'') = \exp(X') \phi_0'(a'),
$$
and the map $a' \mapsto a''$  defines an automorphism of $\alg{A}'[[\lambda]]$ starting at identity, so
$a'' = \exp(\xi)(a')$ for some derivation $\xi \in \lambda\Der(\alg{A}')[[\lambda]]$.

    Regarding (ii), suppose that $X''$ is another derivation as in Proposition~\ref{Deformationofthecommutant}, defining a Poisson structure $\pi''$ on $\alg{A}'[[\lambda]]$; then it follows from the equivalence of solutions in (i) that there is a derivation $\xi\in \lambda \Der(\alg{A}')[[\lambda]]$ satisfying \eqref{eq:equivprobcomm}, which implies that $\exp(\xi)\colon (\alg{A}'[[\lambda]],\pi')\to (\alg{A}'[[\lambda]],\pi'')$ is a Poisson isomorphism, and hence $\pi'$ and $\pi''$ are equivalent.
    \end{proof}

We now discuss conditions for the existence of solutions to Problem~\ref{prob:comm}. We will look for the desired derivation $X'\in\lambda\Der(\alg{B})[[\lambda]]$ inductively. Suppose that $X'_{(k)}$ solves Problem~\ref{prob:comm} up to order $k$, meaning that, for all  $a' \in \alg{A}'$
and $a\in\alg{A}$,
\begin{equation}
    \label{FkPlusOnedefiningrelation}
    \sigma\big(\exp(X'_{(k)})\phi_0'(a'), \Phi(a)\big)
    =
    \lambda^{k+1}F_{k+1}(a', a)
    +
    \cdots,
\end{equation}
for some bilinear map $F_{k+1}\colon \alg{A}' \times \alg{A} \rightarrow \alg{B}$.  Note that $X'_{(0)} = 0$ solves the problem at order $k = 0$,  so it serves as our starting point.

\begin{lemma}\label{lem:closedD}
    For each $a' \in \alg{A}'$, consider the map $D:= F_{k+1}(a',\cdot)\colon\alg{A}\to \alg{B}$.
Then $D\in \CCEder^1(\alg{A}, \alg{B})$ and $\delta D = 0$.
\end{lemma}
\begin{proof}
The fact that $D$ belongs to $\CCEder^1(\alg{A}, \alg{B})$ (i.e., it is a derivation along $\phi_0$) is a direct consequence of the Leibniz rule for $\sigma$.

    Consider $a_1, a_2\in\alg{A}$ and $a' \in\alg{A}'$. We use the Jacobi
    identity for $\sigma$, the defining
    Equation~\eqref{FkPlusOnedefiningrelation} for $F_{k+1}$ and the
    fact that $\Phi$ is a Poisson morphism to compute
    \begin{align*}
        &\sigma\Big(
        \Phi(a_1),
        \sigma\big(\Phi(a_2), \exp(X'_{(k)})\phi'_0(a') \big)
        \Big) \\
        &\quad=
        \sigma\Big(
        \sigma\big(\Phi(a_1), \Phi(a_2)\big),
        \exp(X'_{(k)})\phi'_0(a')
        \Big)
        +
        \sigma\Big(
        \Phi(a_2),
        \sigma\big(\Phi(a_1), \exp(X'_{(k)})\phi'_0(a')\big)
        \Big) \\
        &\quad=
        \sigma\big(\Phi(\pi(a_1, a_2)), \exp(X'_{(k)})\phi'_0(a')\big)
        -
        \sigma\big(
        \Phi(a_2),
        \lambda^{k+1}F_{k+1}(a', a_1)
        +
        \cdots
        \big)\\
        &\quad=
        -
        \lambda^{k+1}F_{k+1}(a', \pi(a_1, a_2))
        +
        \cdots
        -
        \lambda^{k+1}\{\phi_0(a_2), F_{k+1}(a', a_1)\}_\alg{B}
        +
        \cdots \\
        &\quad=
        -
        \lambda^{k+1}F_{k+1}(a', \{a_1, a_2\}_\alg{A})
        -
        \lambda^{k+1}\{\phi_0(a_2), F_{k+1}(a', a_1)\}_\alg{B}
        +
        \cdots.
    \end{align*}
    On the other hand, expanding the left-hand side, we have
    \begin{align*}
        \sigma\Big(
        \Phi(a_1),
        \sigma\big(\Phi(a_2), \exp(X'_{(k)})\phi'_0(a')\big)
        \Big)
        &=
        \sigma\big(
        \Phi(a_1),
        - \lambda^{k+1}F_{k+1}(a',a_2)
        + \cdots
        \big) \\
        &=
        -
        \lambda^{k+1}\{\phi_0(a_1), F_{k+1}(a', a_2)\}_\alg{B}
        +
        \cdots.
    \end{align*}
    Hence, the map $D$ satisfies
    \begin{align*}
        \delta D(a_1, a_2)
        &=
        \{\phi_0(a_1), D(a_2)\}_\alg{B}
        -
        \{\phi_0(a_2), D(a_1)\}_\alg{B}
        -
        D(\{a_1, a_2\}_\alg) \\
        &=
        \{\phi_0(a_1), F_{k+1}(a', a_2)\}_\alg{B}
        -
        \{\phi_0(a_2), F_{k+1}(a', a_1)\}_\alg{B}
        -
        F_{k+1}(a', \{a_1, a_2\}_\alg{A})
        \\
        &= 0.
    \end{align*}
\end{proof}

To construct a solution up to order $k+1$, take $X'_{k+1} \in \Der(\alg{B})$ and notice that
\begin{align*}
    \sigma\big(
    \exp(\lambda^{k+1}X'_{k+1})\exp(X'_{(k)})\phi_0'(a'),
    \Phi(a)
    \big)
    &=
    \sigma\big(
    (\id + \lambda^{k+1}X'_{k+1} + \cdots)
    \exp(X'_{(k)})\phi_0'(a'),
    \Phi(a)
    \big) \\
    &=
    \sigma\big(\exp(X'_{(k)})\phi_0(a'), \Phi(a)\big)
    +
    \lambda^{k+1}\{X'_{k+1}\phi_0'(a'), \phi_0(a)\}_\alg{B}
    +
    \cdots \\
    &=
    \lambda^{k+1}F_{k+1}(a', a)
    +
    \lambda^{k+1}\{X'_{k+1}\phi_0(a'), \phi_0(a)\}_\alg{B}
    +
    \cdots.
\end{align*}
Hence, we need that $X'_{k+1}$ satisfies
\begin{equation}
    \label{equationforXPrime}
    \{\phi_0(a), X'_{k+1}\phi_0'(a')\}_\alg{B}
    =
    F_{k+1}(a', a),
\end{equation}
for all $a' \in \alg{A}'$ and $a \in \alg{A}$. Fixing $a'$ and, letting $b=X'_{k+1}\phi_0'(a')$ and $D=F_{k+1}(a',\cdot)$, this last equation reads
\begin{equation}\label{eq:deltab}
\delta b = D.
\end{equation}

With the set-up of Problem~\ref{prob:comm}, we consider the following additional assumptions:
\begin{itemize}
    \item [(1)] There exists an $\alg{A}'$-linear map $\varphi_1\colon \CCEder^1(\alg{A},
    \alg{B})\to \CCEder^{0}(\alg{A}, \alg{B})$ with the property that $\delta D=0$ implies that $D=\delta \varphi_1(D)$;
    \item[(2)] There is
a horizontal lift along $\phi_0'$ (in the sense of Definition~\ref{horizlift}), i.e., an $\alg{A}'$-linear map $\CCEder^1(\alg{A}',
    \alg{B})\to \CCEder^1(\alg{B},
    \alg{B})$, $D\mapsto D^h$, satisfying $D^h(\phi_0'(a'))=D(a')$ for all $a'\in \alg{A}'$.
\end{itemize}

Note that (1) immediately implies the vanishing of $\HCEder^1(\alg{A}, \alg{B})$; in the geometric setting of interest, the map $\varphi_1$ will naturally arise as part of a chain homotopy.

\begin{proposition}\label{prop:existprobcomm}
    Suppose that (1) and (2) hold. Then  there exists a
    derivation $X' \in \lambda\Der(\alg{B})[[\lambda]]$ such that
    $\exp(X')\phi_0' (\alg{A}'[[\lambda]]) \subseteq C$ (hence solving Problem~\ref{prob:comm})
    \end{proposition}
\begin{proof}
 We must solve \eqref{eq:poiscomm} and, for that, according to the inductive procedure, it is enough to find
    $X'_{k+1} \in \Der(\alg{B})$ satisfying
    \eqref{equationforXPrime}. Using Lemma~\ref{lem:closedD} and (1), we have that  $F_{k+1}$ satisfies, for each fixed $a'\in \alg{A}'$,
    \begin{equation*}
        F_{k+1}(a',\cdot) =\delta (\varphi_1 (F_{k+1}(a',\cdot))).
    \end{equation*}
Comparing with \eqref{eq:deltab}, it suffices to find $X'_{k+1}$ such that $X'_{k+1}(\phi_0'(a')) = \varphi_1(F_{k+1}(a',\cdot))$ for all $a'$.

Just as in Lemma~\ref{lem:closedD}, the Leibniz rule for $\sigma$ in \eqref{FkPlusOnedefiningrelation} ensures that
$$
F_{k+1}(a'_1a_2',\cdot) =  \phi_0'(a_1')F_{k+1}(a_2',\cdot) + \phi_0'(a_2')F_{k+1}(a_1',\cdot),
$$
so, since $\varphi_1$ is $\alg{A}'$-linear, the map $Y\colon \alg{A}'\to \alg{B}$, $a'\mapsto \varphi_1(F_{k+1}(a',\cdot))$, satisfies
$$
Y(a'_1 a'_2) = \phi_0'(a_1') Y(a_2') + \phi_0'(a_2') Y(a_1'),
$$
i.e., $Y\in \CCEder^1(\alg{A}',\alg{B})$. Now (2) guarantees that there exists $X'_{k+1}= Y^h \in \CCEder^1(\alg{B},\alg{B})=\Der(\alg{B})$ such that $X'_{k+1}(\phi_0'(a')) = Y(a') = \varphi_1(F_{k+1}(a',\cdot))$, as desired.
\end{proof}

\subsection{Uniqueness of the construction}\label{subsec:uniquenessdiagram}
We have seen that, starting with a Poisson morphism $\phi_0\colon(\alg{A},\pi_0)\to (\alg{B},\sigma_0)$, solutions to Problems~\ref{Deformationproblem} and \ref{prob:comm} lead to deformations \eqref{eq:deformedpoissonmaps} of the ``bimodule''
\begin{equation}\label{eq:initialbim}
(\alg{A},\pi_0)\stackrel{\phi_0}{\longrightarrow}
(\alg{B},\sigma_0) \stackrel{\phi_0'}{\longleftarrow} (\alg{A}',\pi'_0),
\end{equation}
where we recall that $\alg{A}'$ is the Poisson commutant of $\phi_0(\alg{A})$ in $\alg{B}$ equipped with its natural Poisson structure and $\phi_0'$ the inclusion map. We now discuss simple properties of these ``deformed bimodules''.

For each $i=1,2$, let
\begin{equation}\label{eq:bimodulediag}
(\alg{A}[[\lambda]],\pi^\ind)\stackrel{\Phi^\ind}{\longrightarrow} (\alg{B}[[\lambda]],\sigma^\ind) \stackrel{{\Phi^\ind}'}{\longleftarrow}({\alg{A}}'[[\lambda]],{\pi^\ind}')
\end{equation}
be a diagram of Poisson morphisms with Poisson commuting images whose underlying zeroth order diagram is \eqref{eq:initialbim}.

\begin{proposition}\label{prop:inducedmap}
    Suppose that there are Poisson isomorphisms $\psi\colon (\alg{A}[[\lambda]],\pi^\1)\to (\alg{A}[[\lambda]],\pi^\2)$ and $\Psi\colon (\alg{B}[[\lambda]],\sigma^\1)\to (\alg{B}[[\lambda]],\sigma^\2)$ with
    $\Phi^\2\circ \psi = \Psi \circ \Phi^\1$. Then there is a Poisson isomorphism $\psi'\colon ({\alg{A}}'[[\lambda]],{\pi^\1}')\to ({\alg{A}}'[[\lambda]],{\pi^\2}')$ such that ${\Phi^\2}'\circ \psi'=\Psi\circ {\Phi^\1}'$.
\end{proposition}

\begin{proof}
    Let $C_i$ be the Poisson commutator of $\Phi^\ind(\alg{A}[[\lambda]])$ in $(\alg{B}[[\lambda]],\sigma^\ind)$, which has a Poisson structure given by the restriction of $\sigma^\ind$,  $i=1,2$. As before (c.f. the proof of Proposition~\ref{prop:commimage}), since $\phi_0'$ is injective, the maps ${\Phi^\ind}'\colon ({\alg{A}}'[[\lambda]],{\pi^\ind}')\to C_i$, $i=1,2$, are Poisson isomorphisms. The result follows once we verify that $\Psi(C_1)\subseteq C_2$, since in this case we define $\psi'= ({\Phi^\2}'|_{C_2})^{-1} \circ \Psi \circ {\Phi^\1}'$.

    Now note that if $b$ satisfies $\sigma^\1(b, \Phi^\1(a))=0$ for $a\in \alg{A}[[\lambda]]$, then
    $$
    0=\sigma^\2(\Psi(b),\Psi(\Phi^\1(a)))= \sigma^\2(\Psi(b),\Phi^\2(\psi(a))),
    $$
    which directly implies that if $b\in C_1$, then $\Psi(b)\in C_2$.
    \end{proof}

For the given Poisson morphism $\phi_0\colon(\alg{A},\pi_0)\to (\alg{B},\sigma_0)$, suppose that Problems~\ref{Deformationproblem} and \ref{prob:comm} admit unique solutions (up to equivalence) for any formal Poisson deformations $\pi$ and $\sigma$ of $\pi_0$ and $\sigma_0$, respectively (sufficient conditions are described in Propositions~\ref{Existence}, \ref{Uniqueness} and \ref{prop:existprobcomm}). Then the previous proposition implies the existence of a well-defined map
\begin{equation}\label{eq:algclassmap}
\FPois_{\sigma_0}(\alg{B}) \times    \FPois_{\pi_0}(\alg{A})   \to \FPois_{\pi'_0}(\alg{A}'), \qquad ([\sigma],[\pi])\mapsto [\pi'],
\end{equation}
defined as follows. For representatives $\pi$ and $\sigma$, we choose a solution $\Phi=\exp(X)\phi_0$ to Problem~\ref{Deformationproblem}, then choose a solution $\Phi'=\exp(X')\phi_0'$ to Problem~\ref{prob:comm}, which in turn defines, according to Proposition~\ref{Deformationofthecommutant}, a unique Poisson structure $\pi'$ on $\alg{A}'[[\lambda]]$ such that
$$
(\alg{A}[[\lambda]],\pi)\stackrel{\Phi}{\longrightarrow} (\alg{B}[[\lambda]],\sigma) \stackrel{\Phi'}{\longleftarrow}(\alg{A}'[[\lambda]],\pi')
$$
is a diagram of Poisson maps with commuting images.
Proposition~\ref{prop:unique}, part (ii), along with Proposition~\ref{prop:inducedmap} above, ensure that the equivalence class of $\pi'$ is independent of the choice of representatives $\pi$ and $\sigma$, or specific maps $\Phi$ and $\Phi'$.

%

\section{The geometric set-up}
\label{sec:ClassifyingMap}

In this section we begin our study of Morita equivalence of formal Poisson structures vanishing in zeroth order. The objects of interest are thus equivalence bimodules
\begin{equation}\label{eq:bimformal}
(\Cinf(P_1)[[\lambda]], \pi^\1)
        \stackrel{\Phi^\1}{\longrightarrow}
        (\Cinf(S)[[\lambda]], \omega)
        \stackrel{\Phi^\2}\longleftarrow
        (\Cinf(P_2)[[\lambda]], \pi^\2),
\end{equation}
with $\pi^\1_0 = \pi^\2_0 = 0$. Setting $\lambda=0$, we see that in zeroth order we have a Morita equivalence of the trivial Poisson manifolds $(P_1,0)$ and $(P_2,0)$. In such a case, it is a direct verification that $P_1$ and $P_2$ must be diffeomorphic, so we may as well assume that $P_1=P_2 =P$. It follows that the zeroth order bimodule underlying \eqref{eq:bimformal}  is a self Morita equivalence of $(P,0)$,
$$
(P,0) \stackrel{J_1}{\longleftarrow} (S,\omega_0) \stackrel{J_2}{\longrightarrow} (P,0).
$$
As shown in \cite{BursztynWeinstein-Picard,BursztynFernandes-Picard}, such bimodule is necessarily  isomorphic to one where $S=T^*P$, $J_1=\rho\colon T^*P\to P$ is the natural projection, and
$$
\omega_0 = \omega_{B_0}:= \omega_{\mathrm{can}} + \rho^*B_0,
$$
where $\omega_{\mathrm{can}}$ is the canonical symplectic form on $T^*P$ and $B_0$ is a closed 2-form on $P$. So in order to study equivalence bimodules as in \eqref{eq:bimformal}, we will start with
the Poisson map
\begin{equation}\label{eq:poissonmap}
\rho\colon (T^*P, \omega_{B_0}) \to (P,0)
\end{equation}
and use the results of Section~\ref{sec:defmorphism} to analyze  Problems~\ref{Deformationproblem} and \ref{prob:comm} in this case; i.e., we will set $\alg{A}=C^\infty(P)$ with the zero Poisson bracket, $\alg{B}= C^\infty(T^*P)$ with Poisson bracket $\sigma_0$ determined by $\omega_{B_0}$, and $\phi_0 =  \rho^* \colon \alg{A}\to \alg{B}$.
For this specific geometric example, we will replace Poisson algebras by the corresponding manifolds in the notation, so
we will denote the complex
$(\CCEder^\bullet(\alg{A},\alg{B}), \delta)$ and cohomology $\HCEder^\bullet(\alg{A},\alg{B})$ by $(\CCEder^\bullet(P,T^*P), \delta)$ and $\HCEder^\bullet(P,T^*P)$.

\subsection{Vanishing of cohomology}\label{subsec:vanishing}

To calculate the cohomology $\HCEder^\bullet(P,T^*P)$, let us first spell out the complex $(\CCEder^\bullet(P,T^*P), \delta)$ in our geometric situation. Following Section~\ref{subsec:CEcoh}, elements in $\CCEder^k(P, T^*P)$ are $k$-linear, skewsymmetric maps
$$
D\colon C^\infty(P)\times \ldots \times C^\infty(P)\to C^\infty(T^*P)
$$
satisfying
$$
D(f_1,\ldots,f_l f_l',\ldots,f_k)= \rho^*(f_l) D(f_1,\ldots,f_l',\ldots,f_k) + D(f_1,\ldots,f_l,\ldots,f_k)\rho^*(f_l'),
$$
and hence $\CCEder^k(P, T^*P) = \Gamma(\rho^* \wedge^k TP)$. The differential \eqref{eq:CEdiff} becomes
\begin{align*}
(\delta D)(f_0, \ldots, f_{k}) & =
    \sum_{j=0}^{k} (-1)^{j}
    \big\{    \rho^*(f_j),
    D(f_0, \ldots, \widehat{f_j}, \ldots, f_{k})
    \big\}_{\sigma_0} \\
    & = \sum_{j=0}^{k} (-1)^{j}
    \big\{    \rho^*(f_j),
    D(f_0, \ldots, \widehat{f_j}, \ldots, f_{k})
    \big\}_{can},
\end{align*}
where we have used that the Poisson structure on $P$ is trivial and $\{\rho^*(f),\cdot\}_{\sigma_0} = \{\rho^*(f),\cdot\}_{can}$ for all $f\in C^\infty(P)$ (cf. Lemma~\ref{lem:hamVFB}).

Consider the natural $C^\infty(P)$-module structure on $\CCEder^\bullet(P, T^*P)$ via $\rho^*$.

\begin{lemma}
    \label{tensorialityofdelta}%
    The differential $\delta$ is $\Cinf(P)$-linear.
\end{lemma}
\begin{proof}
For   $D \in \CCEder^k(P, T^*P)$ and $f, f_1, \ldots, f_{k+1} \in \Cinf(P)$, we have
        \begin{align*}
        \delta(fD)(f_1, \ldots, f_{k+1})
        &=
        \sum_{j=1}^{k+1}(-1)^{j+1}
        \left\{
            \rho^*f_j,
            (fD)(f_1, \ldots, \widehat{f_j}, \ldots, f_{k+1})
        \right\}_\canonical \\
        &=
        \sum_{j=1}^{k+1}(-1)^{j+1}
        \left\{
            \rho^*f_j,
            \rho^*f D(f_1, \ldots, \widehat{f_j}, \ldots, f_{k+1})
        \right\}_\canonical \\
        &=
        \sum_{j=1}^{k+1}(-1)^{j+1}
        \left\{
            \rho^*f_j,
            D(f_1, \ldots, \widehat{f_j}, \ldots, f_{k+1})
        \right\}_\canonical \rho^*f \\
        &\quad+
        \sum_{j=1}^{k+1}(-1)^{j+1}
        \underbrace{
          \{\rho^*f_j, \rho^*f\}_\canonical
        }_{=0}
        D(f_1, \ldots, \widehat{f_j}, \ldots, f_{k+1}) \\
        &=
        f(\delta D)(f_1, \ldots, f_{k+1}),
    \end{align*}
    so $\delta(f D) = f\delta D$.
\end{proof}

The vanishing of cohomology is shown in the next result.

\begin{proposition}\label{prop:vanishing}
 For $k\geq 1$, there exists a sequence $h$ of $C^\infty(P)$-linear maps $h_k\colon \CCEder^k(P,T^*P) \to \CCEder^{k-1}(P,T^*P)$ satisfying $\delta h + h \delta = \id$. In particular, $\HCEder^k(P,T^*P)=\{0\}$ for $k\geq 1$.
\end{proposition}

The idea of the proof is that the chain homotopy $h$ will be shown to exist locally, and it will then be globalized using a partition of unity. We start with the local picture.

Let $U\subseteq P$ be an open subset with coordinates $(q^1,\ldots, q^n)$, and with induced coordinates $(q^1,\ldots,q^n,p_1,\ldots,p_n)$ on $T^*U=U\times \mathbb{R}^n$. We will use the multi-index notation $I_k$
for a strictly increasing set of indices $i_1<\ldots<i_k$ of length $k$, so we write $dq^{I_k}:= dq^{i_1}\wedge \ldots \wedge dq^{i_k}$ and likewise for $dp_{I_k}$, $\frac{\partial}{\partial q^{I_k}}$ and $\frac{\partial}{\partial p_{I_k}}$. Then cochains $D\in \CCEder^k(U,T^*U)$ can be uniquely written as
$$
D= \sum_{I_k} D^{I_k}\frac{\partial}{\partial q^{I_k}},
$$
for $D^{I_k}\in C^\infty(T^*U)$. On $T^*U$, denote by  $\Omega^\bullet_V(T^*U)$ the algebra of {\em vertical} differential forms, defined in degree $k$ by forms of type
$$
\eta = \sum_{I_k} \eta^{I_k} dp_{I_k}
$$
for $\eta^{I_k}\in C^\infty(T^*U)$. The {\em vertical} de Rham differential $d_{\mathrm{ver}}$ is given by the usual de Rham differential in the $p_j$ variables, so that on functions $f\in C^\infty(T^*U)$ it acts as $d_{\mathrm{ver}}f =  \sum_j \frac{\partial f}{\partial p_j}dp_j$, and this defines the vertical complex $(\Omega_V(T^*U),d_{\mathrm{ver}})$.

Consider the map $\Psi\colon \CCEder^\bullet(U,T^*U)\to \Omega^\bullet_V(T^*U)$ defined for each degree $k$ by
$$
\sum_{I_k} D^{I_k}\frac{\partial}{\partial q^{I_k}} \mapsto \sum_{I_k} D^{I_k} dp_{I_k}.
$$
Considering the natural $C^\infty(U)$-module structures on both complexes (via $\rho^*\colon C^\infty(U)\to C^\infty(T^*U)$), it is clear that $\Psi$ is an isomorphism of  $C^\infty(U)$-modules.

\begin{lemma}
    \label{Psiisacomplexmap}%
The map $\Psi$ is a cochain map: $d_{\mathrm{ver}}\circ\Psi = \Psi\circ\delta$.
\end{lemma}

\begin{proof}
    For
    $D = \sum_{I_k}D^{I_k}\frac{\partial}{\partial q^{I_k}} \in \CCEder^k(U, T^*U)$, we have
    \begin{equation*}
        \delta D
        =
        \tilde{D}
        =
        \sum_{I_{k+1}}
        \tilde{D}^{I_{k+1}}
        \frac{\partial}{\partial q^{I_{k+1}}},
    \end{equation*}
    where
    \begin{equation*}
        \tilde{D}^{I_{k+1}}
        = \sum_{l=1}^{k+1} (-1)^{l-1} \{q^{i_l}, D(q_{i_1},\ldots,\widehat{q^{i_l}},\ldots, q^{i_{k+1}})\}_{can}  =
        \sum_{l=1}^{k+1}(-1)^{l-1}
        \frac{
          \partial
          D^{i_1 \ldots \widehat{i_l} \ldots i_{k+1}}
        }{\partial p_{i_l}}.
    \end{equation*}
    Therefore
    \begin{equation}
        \label{Psidelta}
        \Psi(\delta D)
        =
        \sum_{I_{k+1}}\left(
            \sum_{l=1}^{k+1}(-1)^{l-1}
            \frac{
              \partial
              D^{i_1 \ldots \widehat{i_l} \ldots i_{k+1}}
            }{\partial p_{i_l}}
        \right)
        dp_{I_{k+1}}.
    \end{equation}
   On the other hand,
    \begin{equation}\label{dPsi}
        d_{\mathrm{ver}}(\Psi(D))
        =
        \sum_{I_k}\left(
            \sum_{r \notin I_k}
            \frac{\partial D^{I_k}}{\partial p_r}dp_r
        \right)
        dp_{I_k}.
    \end{equation}

To prove the lemma,
it suffices to verify that, for each multi-index $I_{k+1}$, the coefficients of $dp_{I_{k+1}}$ in the expressions \eqref{Psidelta} and \eqref{dPsi} coincide. Given $I_{k+1}=(i_1, \ldots, i_{k+1})$, we use the notation $I_k^l$ for the multi-index of length $k$ given by
$(i_1, \ldots, \widehat{i_l}, \ldots, i_{k+1})$,
for $l\in \{1,\ldots,k+1\}$.

For a fixed $I_{k+1}=(i_1, \ldots, i_{k+1})$, in order to find the coefficient of $dp_{I_{k+1}}$ in the expression \eqref{dPsi} of $d_{\mathrm{ver}}(\Psi(D))$  we must collect the terms defined by pairs ($I_k$, $r\notin I_k$) satisfying the condition $I_k\cup \{r\} = I_{k+1}$, since in this case $dp_r dp_{I_k}$ agrees with $dp_{I_{k+1}}$ up to a sign.
But such pairs ($I_k$, $r$) can be equivalently written as $(I_k^l,i_l)$, for $l=1,\ldots,k+1$. Since
$$
\sum_{l=1}^{k+1}\frac{\partial D^{I_k^l}}{\partial p_{i_l}}dp_{i_l}dp_{I_k} = \sum_{l=1}^{k+1} (-1)^{l-1} \frac{\partial D^{I_k^l}}{\partial p_{i_l}}dp_{I_{k+1}},
$$
we see that the coefficients of $dp_{I_{k+1}}$ in \eqref{Psidelta} and \eqref{dPsi} agree, and the result follows.
\end{proof}

We now proceed to the main proof.

\begin{proof}(of Proposition~\ref{prop:vanishing})
The first step is observing that the proposition holds locally.
To see that, let $h_{\mathrm{dR}}\colon \Omega^\bullet(\mathbb{R}^n)\to \Omega^{\bullet-1}(\mathbb{R}^n)$ be the usual de Rham homotopy operator on $\mathbb{R}^n$ (see e.g. \cite{warner2013foundations}), which satisfies $d h_{\mathrm{dR}} + h_{\mathrm{dR}} d = \id$. Take a local chart $U\subseteq P$, so that  $T^*U = U \times \mathbb{R}^n$. We have an induced homotopy operator $h_{\mathrm{ver}}\colon \Omega_V^\bullet(T^*U)\to \Omega_V^{\bullet-1}(T^*U)$ by viewing vertical forms as forms on the fibers $\mathbb{R}^n=\{(p_1,\ldots,p_n)\}$ parametrized by $q=(q^1,\ldots,q^n)\in U$ and
taking $h_{\mathrm{dR}}$ fiberwise, for each fixed $q$. It follows  that $h_{\mathrm{ver}}$ is $C^\infty(U)$-linear and satisfies $d_{\mathrm{ver}}h_{\mathrm{ver}} + h_{\mathrm{ver}}d_{\mathrm{ver}}=\id$.
By Lemma~\ref{Psiisacomplexmap} we have a $C^\infty(U)$-linear isomorphism of complexes $\Psi\colon (\CCEder^\bullet(U,T^*U),\delta)\to (\Omega_V^\bullet(T^*U),d_{\mathrm{ver}})$, which we use to turn $h_\mathrm{ver}$ into a $C^\infty(U)$-linear homotopy operator $h\colon \CCEder^\bullet(U,T^*U) \to \CCEder^{\bullet-1}(U,T^*U)$ as desired.

To prove the global result, consider an atlas $\{U_\alpha\}$ of $P$, and let $\{\varphi_\alpha\}$ be a (locally finite) partition of unity subordinate to it. For each $\alpha$, we have a $C^\infty(U_\alpha)$-linear homotopy operator
$$
h_\alpha\colon \CCEder^\bullet(U_\alpha,T^*U_\alpha) \to \CCEder^{\bullet-1}(U_\alpha,T^*U_\alpha)
$$
satisfying $\delta_\alpha h_\alpha +  h_\alpha \delta_\alpha = \id$. Since the differential
$\delta_\alpha$ is just the restriction of $\delta$ (in the sense that $\delta_\alpha (D|_{U_\alpha})=(\delta D) |_{U_\alpha}$), for each $\alpha$ we have
$$
    \varphi_\alpha \id  = \varphi_\alpha (\delta_\alpha h_\alpha +  h_\alpha \delta_\alpha) = \delta (\varphi_\alpha h_\alpha) + (\varphi_\alpha h_\alpha)\delta
$$
where we have used the $C^\infty(U_\alpha)$-linearity of $\delta_\alpha$ (Lemma~\ref{tensorialityofdelta}).

Now let $h:=\sum_\alpha \varphi_\alpha h_\alpha$ (which is well
    defined since the partition of
    unity is locally finite). It is clear that it is $C^\infty(P)$-linear and $\delta h + h\delta = \id$.
 \end{proof}

\subsection{The classifying action}\label{TheClassifyingMap}

We now collect various consequences of Proposition~\ref{prop:vanishing} concerning formal deformations of the Poisson map \eqref{eq:poissonmap} into a formal equivalence bimodule; we will verify that Problems~\ref{Deformationproblem} and \ref{prob:comm} can be directly solved in this geometric context.

Recall that we are considering the Poisson morphism $\phi_0=\rho^*\colon C^\infty(P)\to C^\infty(T^*P)$ defined by the natural projection $\rho\colon T^*P\to P$, where
$T^*P$ is equipped with the Poisson structure $\sigma_0$ defined by the symplectic form $\omega_{B_0} = \omega_{\mathrm{can}}+\rho^*B_0$, with $B_0$ a closed 2-form on $P$, and $P$ carries the zero Poisson structure. We have the following result.

\begin{theorem} \
    \label{DeformationofthezeroPoissonsymplecticrealization}
     Given any formal symplectic structucture $\omega=\omega_0 + \sum_{k=1}^\infty \lambda^k \omega_k$ on $T^*P$ with $\omega_0=\omega_{B_0}$, and formal Poisson structure $\pi$ on $P$ with $\pi_0=0$, there  exist
    formal vector fields $X, X' \in \lambda\mathfrak{X}(T^*P)[[\lambda]]$ and formal Poisson structure $\pi'$ on $P$, with $\pi'_0=0$,
    such that
    \begin{equation}\label{eq:diagequiv}
    (\Cinf(P)[[\lambda]], \pi)
        \stackrel{\Phi}{\longrightarrow}
        (\Cinf(T^*P)[[\lambda]], \omega)
        \stackrel{\Phi'}\longleftarrow
        (\Cinf(P)[[\lambda]], \pi')
    \end{equation}
    is an equivalence bimodule, where $\Phi = \exp(\mathcal{L}_X)\rho^*$ and $\Phi' = \exp(\mathcal{L}_{X'})\rho^*$.
    \end{theorem}

\begin{proof}
The choice of any horizontal distribution on $T^*P$ (i.e., complementary to the distribution tangent to $\rho$-fibers) determines a horizontal lift operation $\Gamma(\rho^*TP)\to \mathfrak{X}(T^*P)$. (Note that it extends to a $C^\infty(P)$-linear map
$$
\CCEder^\bullet(P,T^*P)=\Gamma(\rho^*\wedge^\bullet TP) \to \mathfrak{X}^\bullet(T^*P) = \CCEder^\bullet(T^*P,T^*P)
$$
so it is a horizontal lift in the sense of Definition~\ref{horizlift}.)
Since $\HCEder^2(P,T^*P)$  vanishes by Proposition~\ref{prop:vanishing}, Proposition \ref{Existence} implies the existence of a formal vector field $X$ such that
$$
\Phi= \exp(\mathcal{L}_X)\rho^*\colon (\Cinf(P)[[\lambda]], \pi) \to (\Cinf(T^*P)[[\lambda]], \omega)
$$ is a Poisson morphism.
Let us fix a choice of $X$ and look at the other leg of the diagram by considering commutators.

The Poisson commutator of $\rho^*C^\infty(P)$ in $C^\infty(T^*P)$, with Poisson structure defined by $\omega_{B_0}$, agrees with itself. So, in the notation of Section~\ref{subsec:commutants}, we have $\alg{A}' = C^\infty(P)$, $\phi'_0=\rho^*$ and $\pi'_0=0$. Let us consider conditions (1) and (2) used in Proposition~\ref{prop:existprobcomm}. In the present case, condition (2)  simply becomes the existence of a horizontal lift $\Gamma(\rho^*TP)\to \mathfrak{X}(T^*P)$, as before. Condition (1), in turn, holds by Proposition~\ref{prop:vanishing}.
As a consequence, Proposition~\ref{prop:existprobcomm} implies the existence of a formal vector field $X'$ such that the image of $\Phi'= \exp(\mathcal{L}_{X'})\rho^*$ Poisson commutes with the image of $\Phi$ with respect to $\omega$. Moreover, by Proposition~\ref{Deformationofthecommutant}, the choice of $X'$ uniquely determines a formal Poisson structure $\pi'$ on $P$ so that $\Phi'\colon (\Cinf(P)[[\lambda]], \pi') \to
        (\Cinf(T^*P)[[\lambda]], \omega)$ is Poisson (or anti-Poisson, with a sign change). This completes the proof.
\end{proof}

The results in Section~\ref{sec:defmorphism} also describe the sense in which $X$, $X'$ and $\pi'$ in the previous theorem are unique.
If  $\overline{X}$ is another formal vector field such that
$\exp(\mathcal{L}_{\overline{X}})\rho^*\colon (\Cinf(P)[[\lambda]], \pi) \to (\Cinf(T^*P)[[\lambda]], \omega)$ is a Poisson morphism, then it must satisfy
$$\exp(\mathcal{L}_{\overline{X}}) =
    \exp(\mathcal{L}_{X_H})\exp(\mathcal{L}_X)\exp(\mathcal{L}_V),
    $$
    with $X_H$ a formal hamiltonian vector field with respect to $\omega$ and $V\in \lambda \mathfrak{X}(T^*P)[[\lambda]]$ vertical (i.e., $V\circ\rho^* = 0$). This follows from Propositions~\ref{Uniqueness} and \ref{prop:vanishing}. On the other hand, another formal vector field $\overline{X}'$ such that $\exp(\mathcal{L}_{X'})\rho^*$ has image in the commutator of $\Phi$ must satisfy
    $$\exp(\mathcal{L}_{\overline{X}'})\rho^* = (\exp(\mathcal{L}_{X'})\rho^*)\circ \exp(Y),
    $$ for some $Y\in \lambda \mathfrak{X}(P)[[\lambda]]$, by Proposition~\ref{prop:unique}. A given $X'$ determines a unique formal Poisson structure $\pi'$ on $P$ for which $\Phi'\colon  (\Cinf(P)[[\lambda]],\pi') \to (\Cinf(T^*P)[[\lambda]],\omega)$ is a Poisson map by Proposition~\ref{Deformationofthecommutant}, and its equivalence is independent of the choice of $X'$, as shown in Proposition~\ref{prop:unique}.

 As seen in Section~\ref{subsec:uniquenessdiagram}, see \eqref{eq:algclassmap}, the result in Theorem~\ref{DeformationofthezeroPoissonsymplecticrealization} and its uniqueness properties lead to a natural map
 \begin{equation*}
\FPois_{\sigma_{B_0}}(T^*P) \times \FPois_0(P) \to \FPois_0(P),
 \end{equation*}
 where $\sigma_{B_0}$ is the Poisson structure associated with $\omega_{B_0}$. Since this map depends on an initial choice
  of symplectic form $\omega_{B_0}=\omega_{\mathrm{can}}+\rho^*B_0$ on $T^*P$, we denote it by $\gamma_{B_0}$ to make the dependence on $B_0$ explicit.
 By means of the identification $\FPois_{\sigma_{B_0}}(T^*P)=\lambda \HdR(T^*P)[[\lambda]]$ from Remark~\ref{rem:symplecticdeformations} and the isomorphism $\HdR^2(P)\cong \HdR^2(T^*P)$ given by pullback by $\rho$, we write
 \begin{equation} \label{eq:classmap1}
\gamma_{B_0}\colon \lambda \HdR^2(P)[[\lambda]] \times \FPois_0(P) \to \FPois_0(P).
 \end{equation}
Concretely, this map is described as follows: given representatives $\sum_{k=1}^\infty\lambda^k B_k$ and $\pi$ of classes in $\lambda \HdR^2(P)[[\lambda]]$ and $\FPois_0(P)$, the resulting class in $\FPois_0(P)$ is defined by any  formal Poisson structure $\pi'$ fitting into an equivalence bimodule \eqref{eq:diagequiv}, with
$$
 \omega = \omega_{B_0} + \rho^*\Big(\sum_{k=1}^\infty\lambda^k B_k \Big) = \omega_{\mathrm{can}} + \rho^*B_0 + \rho^*\Big(\sum_{k=1}^\infty\lambda^k B_k \Big).
$$
As we remarked,  the closed 2-form $B_0$ is fixed in this construction (while $B_k$, $k\geq 1$, are only considered up to exact forms).
But we have the following observation.

\begin{lemma}
    The map $\gamma_{B_0}$ only depends on the cohomology class of $B_0$.
\end{lemma}

\begin{proof}
If $B_0-B_0' = d\theta$ for $\theta\in \Omega^1(P)$, then fiber-translation by $\theta$ defines a symplectomorphism $\psi_\theta\colon (T^*P,\omega_{B_0})\to (T^*P,\omega_{B_0'})$ such that $\rho\circ \psi_\theta = \rho$.  So for any closed $b\in \lambda\Omega^2(P)[[\lambda]]$, we have
$$
\psi_\theta^*(\omega_{B_0'}+ \rho^*b)= \omega_{B_0}+ \rho^*b.
$$
It directly follows that, if $\pi$ and $\pi'$ fit into an equivalence bimodule \eqref{eq:diagequiv} with
$\omega = \omega_{B_0}+ \rho^*b$,  then they also fit into an equivalence bimodule with formal symplectic form $\omega_{B_0'}+ \rho^*b$ on $T^*P$, showing that $\gamma_{B_0}([b],[\pi])=\gamma_{B_0'}([b],[\pi])=[\pi']$.
\end{proof}

As a consequence of the previous lemma, we see that the map \eqref{eq:classmap1} gives rise to a well-defined map
\begin{equation}\label{eq:classaction}
\gamma\colon \HdR^2(P)[[\lambda]]\times \FPois_0(P)\to \FPois_0(P), \;\;\;\; \gamma([B],[\pi])=\gamma_{B_0}([b],[\pi]),
\end{equation}
 where $B=\sum_{k=0}^\infty \lambda^k B_k$ and $b= \sum_{k=1}^\infty \lambda^k B_k$. It is also evident from the way this map is defined that if $\pi^B$ is a representative of the class $\gamma([B],[\pi])$, then $\pi^B$ and $\pi$ are Morita equivalent.

 It turns out that $\gamma$ actually defines an {\em action}
of the abelian group $\HdR(P)[[\lambda]]$ on the set $\FPois_0(P)$ (this can be verified as a consequence of Theorem~\ref{mainresult} below). We call $\gamma$ the {\bf classifying action}, since its orbits classify Morita equivalent formal Poisson structures in
$\FPois_0(P)$, in the sense described by the next result.
Recall that the natural action of the group of diffeomorphisms $\Diff(P)$ on  formal Poisson structures on $P$ (by pushforward at each order) descends to an action of
$\Diff(P)$ on $\FPois_0(P)$, $\psi_*([\pi])=[\psi_*\pi]$.

\begin{theorem}
    \label{CharacterizationofMEFPSviadeformation}%
Two formal Poisson structures $\pi$   and $\pi'$ on $P$, with $\pi_0=\pi'_0=0$, are Morita equivalent if and only if their classes in $\FPois_0(P)$ satisfy
    $[\pi'] = \psi_* \gamma([B],[\pi])$ for some $[B] \in \HdR^2(P)[[\lambda]]$
    and $\psi\in\Diff(P)$.
\end{theorem}
\begin{proof}
If $\pi^B$ is a representative of $\gamma([B],[\pi])$, then
it is Morita equivalent to $\pi$. Since a representative $\pi'$ of $\psi_* \gamma([B],[\pi])$ is Poisson isomorphic to $\pi^B$,  it is also Morita equivalent to $\pi$.


For the converse, suppose that $\pi$ and $\pi'$ are Morita
equivalent by means of an equivalence bimodule
    \begin{equation*}
        (\Cinf(P)[[\lambda]], \pi)
        {\longrightarrow}
        (\Cinf(S)[[\lambda]], \omega)
        {\longleftarrow}
        (\Cinf(P)[[\lambda]], \pi'),
    \end{equation*}
with underlying zeroth order equivalence bimodule
$(P, 0) \stackrel{J_1}{\longleftarrow} (S, \omega_0)
\stackrel{J_2}{\longrightarrow} (P, 0)$.
As shown in \cite[Sec.~6.2]{BursztynWeinstein-Picard} (see also \cite{BursztynFernandes-Picard}), we have identifications
$S = T^*P$,
$\omega_0 = \omega_\canonical + \rho^*B_0$, for some closed
$B_0 \in \Omega^2(P)$, $J_1 = \rho$, and $J_2 = \psi\circ \rho$,
for some $\psi \in \Diff(P)$. Since each $\omega_k$, $k\geq 1$, is closed, it is cohomologous to $\rho^*B_k$, for a closed $B_k\in \Omega^2(P)$, and hence $\omega$ is cohomologous to a formal symplectic form
$$\omega_B = \omega_0 + \rho^*\Big(\sum_{k=1}^\infty\lambda^k B_k\Big)= \omega_\canonical+\rho^*B,$$
with $B=B_0+ \sum_{k=1}^\infty\lambda^k B_k\in\Omega^2(P)[[\lambda]]$ closed. It follows (c.f. Remark~\ref{rem:symplecticdeformations}) that there exists $Z\in\lambda\mathfrak{X}(P)[[\lambda]]$ such that
\begin{equation*}
        \exp(\mathcal{L}_Z)\colon
        (\Cinf(T^*P)[[\lambda]], \omega)
        \to
        (\Cinf(T^*P)[[\lambda]], \omega_B)
\end{equation*}
preserves Poisson brackets. As a result, $\pi$ and $\pi'$ fit into an equivalence bimodule of the form
$$(\Cinf(P)[[\lambda]], \pi)
        \stackrel{\Phi}{\longrightarrow}
        (\Cinf(T^*P)[[\lambda]], \omega_B)
        \stackrel{\Psi}{\longleftarrow}
        (\Cinf(P)[[\lambda]], \pi'),$$
with $\Phi_0=\rho^*$ and $\Psi_0=\rho^*\circ \psi^*$. Then
the following is also an equivalence bimodule:
$$(\Cinf(P)[[\lambda]], \pi)
        \stackrel{\Phi}{\longrightarrow}
        (\Cinf(T^*P)[[\lambda]], \omega_B)
        \stackrel{\Phi'}{\longleftarrow}
        (\Cinf(P)[[\lambda]], \psi^*\pi'),$$
for $\Phi'= \Psi \circ (\psi^{-1})^*$, but now $\Phi'_0 = \rho^*$. So this last bimodule is exactly like the one in \eqref{eq:diagequiv} (see Lemma~\ref{lem:formalVF}), showing that
$[\psi^* \pi'] = \gamma([B],[\pi])$, or $[\pi'] = \psi_* \gamma([B],[\pi])$.
\end{proof}

\section{Description of the classifying action via \emph{B}-fields}
\label{sec:B-field_classification}

\subsection{The main result}
We now prove our main result, stated in Theorem~\ref{thm:main}, relating Morita equivalence to the action of $B$-fields. The last ingredient that we need is the existence of suitable self equivalence bimodules for formal Poisson structures vanishing in zeroth order:

\begin{lemma}\label{lem:selfequiv}
Any formal Poisson structure $\pi$ on $P$, with $\pi_0=0$, admits a self equivalence bimodule of the following type:
\begin{equation*}
    (\Cinf(P)[[\lambda]], \pi)
    \xrightarrow[]{\Phi}
    (\Cinf(T^*P)[[\lambda]], \omega)
    \xleftarrow[]{\Phi'}
    (\Cinf(P)[[\lambda]], \pi),
\end{equation*}
where $\omega = \omega_\canonical + \sum_{k=1}^\infty \lambda^k d\theta_k$, $\Phi=\rho^*$ and $\Phi' = \exp(\mathcal{L}_Y)\rho^*$.
\end{lemma}

This lemma will be a direct consequence of Proposition~\ref{prop:selfequiv}, proven in the next subsection.
We will assume it here to prove our main result.

Recall that, given a $B$-field, i.e., a formal series $B=\sum_{k=0}^\infty \lambda^k B_k \in \Omega^2(P)[[\lambda]]$ of closed 2-forms, and a formal Poisson structure $\pi$ with $\pi_0=0$, we can define a new formal Poisson structure $\tau_B(\pi)$, as explained in Section~\ref{subsec:MEBfields}, see \eqref{eq:tauB} (the condition $\pi_0=0$
makes the necessary invertibility of $\id + B_0^\flat\circ \pi^\sharp_0$ automatic). This defines an action of $\HdR^2(P)[[\lambda]]$ on $\FPois_0(P)$, see \eqref{eq:Bfieldact}.

\begin{theorem}
    \label{mainresult}%
    The classifying action $\gamma$ is given by $\gamma([B],[\pi])=[\tau_{-B}(\pi)]$.
\end{theorem}

\begin{proof}
Take representatives $B$ and $\pi$ of classes in $\HdR^2(P)[[\lambda]]$ and $\FPois_0(P)$, consider
the formal Poisson structure $\tau_{-B}(\pi)$ and a
self-equivalence bimodule
$$
(\Cinf(P)[[\lambda]], \tau_{-B}(\pi))
    \xrightarrow[]{\rho^*}
    (\Cinf(T^*P)[[\lambda]], \omega)
    \xleftarrow[]{\Phi'}
    (\Cinf(P)[[\lambda]], \tau_{-B}(\pi)),
$$
as in Lemma~\ref{lem:selfequiv}.
By Theorem~\ref{Bfieldactionondualpairs}, the following is also an equivalence bimodule:
    \begin{equation*}
        (\Cinf(P)[[\lambda]], \pi)
        \xrightarrow[]{\rho^*}
        (\Cinf(T^*P)[[\lambda]], \omega+\rho^*B)
        \xleftarrow[]{\Phi'}
        (\Cinf(P)[[\lambda]], \tau_{-B}(\pi)).
    \end{equation*}
    By Lemma~\ref{lem:selfequiv} $\omega$ is cohomologous to $\omega_\canonical$, so $\omega +\rho^*B$ is cohomologous to $\omega_\canonical +\rho^*B$, and there is a Poisson isomorphism $\exp(\mathcal{L}_Z)\colon (\Cinf(T^*P)[[\lambda]],  \omega+\rho^*B)\to(\Cinf(T^*P)[[\lambda]],  \omega_\canonical+\rho^*B)$. We then obtain an equivalence bimodule
    \begin{equation*}
        (\Cinf(P)[[\lambda]], \pi)
        \xrightarrow[]{\widehat{\Phi}}
        (\Cinf(T^*P)[[\lambda]], \omega_\canonical +\rho^*B)
        \xleftarrow[]{\widehat{\Phi}'}
        (\Cinf(P)[[\lambda]], \tau_{-B}(\pi)),
    \end{equation*}
    where $\widehat{\Phi} = \exp(\mathcal{L}_Z)\rho^*$ and $\widehat{\Phi}' = \exp(\mathcal{L}_Z)\Phi'$. By the very definition of $\gamma$, this means that $\tau_{-B}(\pi)$ is a representative of the class $\gamma([B],[\pi])$.
  \end{proof}

In conclusion, the action \eqref{eq:Bfieldact} of $B$-fields on $\FPois_0(P)$ agrees (up to a conventional sign) with the classifying action \eqref{eq:classaction},
so Theorem~\ref{thm:main} is a direct consequence of this result together with Theorem~\ref{CharacterizationofMEFPSviadeformation}.

%
%

\subsection{Proof of Lemma~\ref{lem:selfequiv}}
It will be convenient to use a Dirac-geometric viewpoint to Poisson structures
(see e.g. \cite{Bursztyn2013-DiracManifolds,Meinrenken2018PoissonfromDiracPerspective}), suitably adapted to the formal context.

For a smooth manifold $M$, consider the bundle $\mathbb{T}M := TM\oplus T^*M$ and its space of smooth sections $\Gamma(\mathbb{T}M)=\mathfrak{X}(M)\oplus \Omega^1(M)$ equipped with the $C^\infty(M)$-bilinear pairing
$$
\langle X\oplus\xi,Y \oplus \eta\rangle := \iota_X\eta+\iota_Y\xi
$$
and the Courant-Dorfman bracket
$$
\Cour{X\oplus\alpha,Y\oplus\beta} := [X, Y]\oplus (\mathcal{L}_{X}\beta -
\iota_Yd\alpha).
$$
These structures define the canonical Courant-algebroid structure on $\mathbb{T}M$
 \cite{Courant1990,LWX1997-ManinTriples}, with anchor map given by the projection $\mathbb{T}M = TM\oplus T^*M \to TM$.

Here we will consider the same structures on $\Gamma(\mathbb{T}M)[[\lambda]] = \mathfrak{X}(M)[[\lambda]]\oplus
\Omega^1(M)[[\lambda]]$,
viewed as a $\Cinf(M)[[\lambda]]$-module.
The definitions of $\langle \cdot,\cdot \rangle$ and $\Cour{\cdot,\cdot}$ are given by the same formulas, extended by linearity in the formal parameter $\lambda$ (see \cite[Section~3.2]{BDW-Characteristicclasses-starproducts}).
These objects can be used to describe Poisson structures as in the usual case. A formal bivector field $\pi=\sum_{k=0}^\infty \lambda^k \pi_k$ defines a submodule of $\Gamma(\mathbb{T}M)[[\lambda]]$ given by its graph,
$$
\mathrm{gr}(\pi)= \{(\pi^\sharp(\alpha),\alpha)\,|\, \alpha\in \Omega^1(M)[[\lambda]]\}.
$$
The pairing $\langle \cdot,\cdot \rangle$ vanishes on $\mathrm{gr}(\pi)$, and the condition on $\pi$ being Poisson is equivalent to $\mathrm{gr}(\pi)$ being involutive for the Courant-Dorfman bracket.

\subsubsection*{Symmetries of the formal Courant structure}
We need to collect some facts about symmetries of $\Gamma(\mathbb{T}M)[[\lambda]]$, which are direct adaptations of the results e.g. in \cite[Sec.~2.1]{Gualtieri2011} (see also \cite{Shengda2009HamiltonianSymmetries}) for the standard Courant algebroid $\mathbb{T}M$.
A {\em symmetry}, or {\em automorphism}, of  $\Gamma(\mathbb{T}M)[[\lambda]]$ is a pair $(F,\phi)$, where $\phi\colon C^\infty(M)[[\lambda]]\to C^\infty(M)[[\lambda]]$ is an isomorphism of commutative algebras and
\begin{equation*}
        F\colon
        \Gamma(\mathbb{T}M)[[\lambda]]
        \rightarrow
        \Gamma(\mathbb{T}M)[[\lambda]]
    \end{equation*}
    is a $\mathbb{K}[[\lambda]]$-linear map preserving the relevant structures: for $e, e_1, e_2 \in\Gamma(\mathbb{T}M)[[\lambda]],$ and
    $g \in \Cinf(M)[[\lambda]]$, $F(ge) = \phi^{-1}(g)F(e)$,
    $\phi\langle F(e_1), F(e_2)\rangle = \langle e_1,
        e_2 \rangle$, and $\Cour{F(e_1), F(e_2)} = F(\Cour{e_1, e_2})$. A class of examples is given by formal diffeomorphisms of $M$: for $X\in \lambda \mathfrak{X}(M)[[\lambda]]$, we let
        $$
        F(Y\oplus\beta)= \exp(-\mathcal{L}_X)(Y)\oplus \exp(-\mathcal{L}_X)(\beta),\qquad  \phi=\exp(\mathcal{L}_X)
        $$
We denote this automorphism by $ \exp(-\mathcal{L}_X)$, if there is no risk of confusion. Another type of example is given by $B$-fields: for a closed $B\in \Omega^2(M)[[\lambda]]$, we take
$$
F(Y\oplus \beta) = Y\oplus (\beta + i_YB),\qquad \phi=\id.
$$
This automorphism is denoted by $\tau_B$, since its effect on formal Poisson structures agrees with gauge transformations
(in the sense that $\tau_B(\mathrm{gr}(\pi))= \mathrm{gr}(\tau_B(\pi))$). We will be interested here in automorphisms with $\phi=\exp(\mathcal{L}_X)$ for some $X\in \lambda \mathfrak{X}(M)[[\lambda]]$; analogously to \cite[Prop.~2.2]{Gualtieri2011}, such symmetries are necessarily given by compositions of formal diffeomorphisms and $B$-fields.

An {\em infinitesimal symmetry}, or {\em derivation}, of $\Gamma(\mathbb{T}M)[[\lambda]]$ is a pair $(D, X)$, where
    \begin{equation*}
        D\colon
        \Gamma(\mathbb{T}M)[[\lambda]]
        \rightarrow
        \Gamma(\mathbb{T}M)[[\lambda]]
    \end{equation*}
    is a $\mathbb{K}[[\lambda]]$-linear map and
    $X \in \mathfrak{X}(M)[[\lambda]]$ such that, for $e, e_1, e_2 \in\Gamma(\mathbb{T}M)[[\lambda]]$ and
    $g \in \Cinf(M)[[\lambda]]$, $D(g e) = gD(e) + \mathcal{L}_X(g)e$, $D\Cour{e_1,e_2} = \Cour{D(e_1),  e_2} + \Cour{e_1, D(e_2)}$, and
    $\mathcal{L}_X\langle e_1, e_2 \rangle = \langle
        D(e_1), e_2\rangle+\langle e_1, D(e_2)\rangle$. The following are two key examples: any $X \in \mathfrak{X}(M)[[\lambda]]$ defines a derivation $(D,X)$ by $D(Y\oplus\beta) = [X,Y]\oplus \mathcal{L}_X\beta$, while any closed $b\in \Omega^2(M)[[\lambda]]$ defines a derivation $(D,0)$ with $D(Y\oplus \beta)= - 0\oplus i_Yb$.

        Derivations arise as infinitesimal generators of 1-parameter subgroups of automorphisms $(F_t,\phi_t)$, $D=-\frac{d}{dt}\Big |_{t=0} F_t$, and from this perspective the infinitesimal counterparts of symmetries $(F,\phi)$ with $\phi$ a formal diffeomorphism are the derivations $(D,X)$ with symbol $X\in \lambda \mathfrak{X}(M)[[\lambda]]$ vanishing in zeroth order. Similarly to \cite[Sec.~2.1]{Gualtieri2011}, one can check that all such derivations are given by pairs $(X,b) \in \lambda \mathfrak{X}(M)[[\lambda]]\oplus \Omega^2(M)[[\lambda]]$ with $db=0$, acting by the sum of the effects of $X$ and $b$:
        \begin{equation}\label{eq:bder}
        (X,b) (Y\oplus \beta) = [X,Y]\oplus (\mathcal{L}_X\beta - i_Y b).
        \end{equation}
        The corresponding 1-parameter subgroup of automorphisms of $\Gamma(\mathbb{T}M)[[\lambda]]$ is explicitly given by
        \begin{equation}\label{eq:1param}
        F_t = \exp(-t\mathcal{L}_X)\tau_{B_t}, \;\;\; \mbox{ where }
        B_t =  \int_{0}^t\exp(s\mathcal{L}_X)bds.
        \end{equation}

\subsubsection*{The bimodule condition}

Let us consider the diagram
\begin{equation}
\label{eq:dualpairdiagram}
(\Cinf(P_1)[[\lambda]], \pi^\1)
\xrightarrow[]{\Phi^\1}
(\Cinf(S)[[\lambda]], \omega)
\xleftarrow[]{\Phi^\2}
(\Cinf(P_2)[[\lambda]], \pi^\2),
\end{equation}
where $\pi^\1$ and $\pi^\2$ are formal Poisson structures,
$\omega$ is a formal symplectic structure, and
$\Phi^\ind = \exp(Z^\ind)J_i^*$ for $Z^\ind \in \lambda\mathfrak{X}(S)[[\lambda]]$ and $J_i\colon S\to P$ a surjective submersion, for $i=1,2$.  Following \cite{Frejlich}, we will describe a convenient criterion ensuring that this diagram is a bimodule, in the sense of Section~\ref{subsec:formal}.

For a surjective submersion $J\colon S\rightarrow P$ and a formal Poisson structure $\pi$ on $P$, we will denote by
$$
J^!\mathrm{gr}(\pi) \subseteq \mathfrak{X}(S)[[\lambda]]\oplus\Omega^1(S)[[\lambda]]
$$
the $C^\infty(S)[[\lambda]]$-submodule defined by elements $X\oplus \alpha$ satisfying the following pointwise condition:
for each $z\in S$, with $q = J(z)$, there exists $\xi\in T_q^*P[[\lambda]]$ such that
\begin{equation}\label{eq:bimagedef}
(T_zJ)^*\xi = \alpha_z\;\;\;  \mbox{ and }\;\;
		T_zJ(X_z) = \pi^\sharp_{q}(\xi).
\end{equation}
Note that this is a natural adaptation to the formal context of the notion of ``backward image'' of Dirac structures, see e.g. \cite{BursztynRadko2003,Bursztyn2013-DiracManifolds}. One of its main properties, verified as in the usual context, is that $J^!\mathrm{gr}(\pi)$ is closed under the Courant-Dorfman bracket (see e.g. \cite[Prop.~5.6]{Bursztyn2013-DiracManifolds}, \cite[Prop.~2.13]{Meinrenken2018PoissonfromDiracPerspective}).
For $\Phi=\exp(\mathcal{L}_Z)J^*$, with $Z\in \lambda \mathfrak{X}(S)[[\lambda]]$, we define
$$
\Phi^!\mathrm{gr}(\pi):= \exp(\mathcal{L}_Z)J^!\mathrm{gr}(\pi).
$$

The main observation is the following (c.f. \cite{Frejlich}):

\begin{lemma}
    \label{lem:bimodulecriterion}%
    Suppose that in diagram ~\eqref{eq:dualpairdiagram} we have $\mathrm{dim}(S)=\mathrm{dim}(P_1) + \mathrm{dim}(P_2)$
    and
    \begin{equation}
    \label{eq:Diracstructuresrelation}
        \tau_{-\omega}((\Phi^\1)^!\mathrm{gr}(\pi^\1)) = (\Phi^\2)^!\mathrm{gr}(\pi^\2).
    \end{equation}
Then diagram ~\eqref{eq:dualpairdiagram} is a bimodule.
\end{lemma}

\begin{proof}
    The assertion that \eqref{eq:dualpairdiagram} is a bimodule means that (1) $\Phi^\1$ is a Poisson morphism, (2) $\Phi^\2$ is anti-Poisson, and (3) their images Poisson commute. Note first that condition \eqref{eq:Diracstructuresrelation} is equivalent to $\tau_{\omega}((\Phi^\2)^!\mathrm{gr}(\pi^\2)) = (\Phi^\1)^!\mathrm{gr}(\pi^\1)$, so if the assumptions in the lemma imply that $\Phi^\1$ is Poisson, by changing the roles of $\Phi^\1$ and $\Phi^\2$ they also imply that $\Phi^\2$ is anti-Poisson. Hence it suffices to check that (1) and (3) hold.

    Since we can always change the diagram \eqref{eq:dualpairdiagram} by a formal diffeomorphism on $S$ (as in \eqref{eq:changeS}) and the validity of the lemma is independent of this change, there is no loss in generality in assuming that $\Phi^\1=J_1^*$. We will make this assumption and write $\Phi^\2=\exp(\mathcal{L}_Z)J_2^*$, so \eqref{eq:Diracstructuresrelation} reads
    \begin{equation}\label{eq:diracrel}
    \tau_{-\omega}(J_1^! \mathrm{gr}(\pi^\1))= \exp(\mathcal{L}_Z) J_2^!\mathrm{gr}(\pi^\2)).
    \end{equation}

    The main observation for the proof of the lemma is that the map
    $\omega^\flat\circ \exp(\mathcal{L}_Z): \mathfrak{X}(S)[[\lambda]]\to \Omega^1(S)[[\lambda]]$ restricts to an isomorphism
    \begin{equation}\label{eq:isom1}
      \omega^\flat\circ \exp(\mathcal{L}_Z)\colon \Gamma(\ker(TJ_2))[[\lambda]] \to \Gamma(\mathrm{Ann}(\ker(TJ_1)))[[\lambda]],
    \end{equation}
    while $\exp(-\mathcal{L}_Z)\circ \omega^\flat$ restricts to an isomorphism
 \begin{equation}\label{eq:isom2}
      \exp(-\mathcal{L}_Z)\circ \omega^\flat\colon \Gamma(\ker(TJ_1))[[\lambda]] \to \Gamma(\mathrm{Ann}(\ker(TJ_2)))[[\lambda]].
    \end{equation}
Let us verify that \eqref{eq:isom1} is an isomorphism. Let $Y \in \mathfrak{X}(S)[[\lambda]]$ be such that $TJ_2(\exp(-\mathcal{L}_Z)Y)=0$. Then $\exp(-\mathcal{L}_Z)Y \in J_2^!\mathrm{gr}(\pi^\2)$ (c.f. \eqref{eq:bimagedef}), so $Y\in \exp(\mathcal{L}_Z)J_2^!\mathrm{gr}(\pi^\2)$. By \eqref{eq:diracrel}, $Y\oplus i_Y\omega \in J_1^!\mathrm{gr}(\pi^\1)$, which implies that, at each point $z\in S$, $i_Y\omega_z = J_1^*\xi$ for $\xi \in T^*_{J_1(z)}P_1[[\lambda]]$. Hence $\omega^\flat(Y)\in\Gamma(\mathrm{Ann}(\ker(TJ_1)))[[\lambda]]$, so we obtain the map \eqref{eq:isom1}. By the dimension condition in the lemma, the
injective map $\omega_0^\flat\colon \ker(TJ_2)\to \mathrm{Ann}(\ker(TJ_1))$ is an isomorphism.  So  \eqref{eq:isom1}
is an isomorphism in zeroth order, hence it is an isomorphism.
The verification that \eqref{eq:isom2} is an isomorphism is analogous.

We now check that \eqref{eq:diracrel} implies that $\Phi^\1=J_1^*$ is a Poisson morphism. Take $\beta\in \Omega^1(P_1)[[\lambda]]$, and let $X\in \mathfrak{X}(S)[[\lambda]]$ satisfy $i_X\omega=J_1^*\beta$.
We must verify that, at each point,
\begin{equation}\label{eq:poismap}
TJ_1(X)=(\pi^\1)^\sharp(\beta).
\end{equation}
Since $\omega^\flat(X)\in \Gamma(\mathrm{Ann}(\ker(TJ_1)))$, by the isomorphism \eqref{eq:isom1} we know that
$$
\exp(-\mathcal{L}_Z)X\in \Gamma(\ker(TJ_2))[[\lambda]] \subseteq J_2^!\mathrm{gr}(\pi^\2).$$
Hence
$$
X\in \exp(\mathcal{L}_Z)J_2^!\mathrm{gr}(\pi^\2)=\tau_\omega(J_1^!\mathrm{gr}(\pi^\1)),
$$
which means that $X\oplus i_X\omega \in J_1^!\mathrm{gr}(\pi^\1)$. Since $i_X\omega=J_1^*\beta$ (and $\beta$ is unique since $J_1$ is a submersion), this implies that \eqref{eq:poismap} holds (c.f. \eqref{eq:bimagedef}).

To verify that the images of $\Phi^\1=J_1^*$ and $\Phi^\2=\exp(\mathcal{L}_Z)J_2^*$ Poisson commute, recall that
the Hamiltonian vector field $X_{\Phi^\2(g)}$ satisfies $\omega^\flat(X_{\Phi^\2(g)})=d \Phi^\2(g) =  \exp(\mathcal{L}_Z)J_2^*dg$, so
$$
\exp(-\mathcal{L}_Z)\omega^\flat(X_{\Phi^\2(g)}) \in \Gamma(\mathrm{Ann}(\ker(TJ_2)))[[\lambda]].
$$
By the isomorphism \eqref{eq:isom2}, $X_{\Phi^\2(g)}\in \Gamma(\ker(TJ_1))[[\lambda]]$, and so
$$
\{\Phi^\1(f), \Phi^\2(g) \}_\omega= -J_1^*df (X_{\Phi^\2(g)})=0.
$$
\end{proof}

\subsubsection*{The construction of the self equivalence bimodule}
Let $\pi$ be a formal Poisson structure on a manifold $P$ that vanishes in zeroth order. We now  have the ingredients to prove the existence of self-equivalence bimodules
$$
(\Cinf(P)[[\lambda]], \pi)
    \xrightarrow[]{}
    (\Cinf(T^*P)[[\lambda]], \omega)
    \xleftarrow[]{}
    (\Cinf(P)[[\lambda]], \pi)
$$
as in Lemma~\ref{lem:selfequiv}.

Fix a linear connection $\nabla$ on the cotangent bundle $\rho\colon T^*P \to P$, denote by $\mathrm{hor}$ the corresponding horizontal lift,
and consider the formal vector field $Z\in \lambda \mathfrak{X}(T^*P)[[\lambda]]$ defined at $\xi\in T^*P$ by
$$
Z_\xi=\mathrm{hor}_\xi(\pi^\sharp(\xi)).
$$
\begin{proposition}\label{prop:selfequiv}
The following is an equivalence bimodule:
\begin{equation}
     (\Cinf(P)[[\lambda]], \pi)
    \xrightarrow[]{\rho^*}
    (\Cinf(T^*P)[[\lambda]], \omega)
    \xleftarrow[]{\exp(\mathcal{L}_Z)\rho^*}
    (\Cinf(P)[[\lambda]], \pi),
\end{equation}
where $\omega :=
        \int_{0}^1\exp(s\mathcal{L}_Z)\omega_\canonical ds.$

\end{proposition}

\begin{proof}
    Let $\theta_\canonical\in\Omega^1(T^*P)$ be the
    tautological $1$-form on $T^*P$, so that $\omega_\canonical = -d\theta_\canonical$. The inner derivation $\Cour{Z\oplus \theta_\canonical,\cdot}$ of $\Gamma(\mathbb{T}M)[[\lambda]]$  coincides with the derivation defined by the pair $(Z,d\theta_\canonical)$ as in \eqref{eq:bder}. So it generates a 1-parameter subgroup of automorphisms $F_t= \exp(-t \mathcal{L}_Z)\tau_{B_t}$, where
    $$
    B_t = -\int_{0}^t\exp(s\mathcal{L}_Z)\omega_\canonical ds.
    $$

Recall that $(\theta_\canonical)_\xi = \rho^*\xi$ and $T\rho(Z_\xi)=\pi^\sharp(\xi)$, so
    $Z \oplus \theta_\canonical\in \rho^!\mathrm{gr}(\pi)$. From the involutivity of $\rho^!\mathrm{gr}(\pi)$ with respect to the Courant-Dorfman bracket, it follows that the derivation $\Cour{Z\oplus \theta_\canonical,\cdot}$ preserves $\rho^!\mathrm{gr}(\pi)\subseteq \Gamma(\mathbb{T}M)[[\lambda]]$, hence so does its corresponding flow:
    $$
    F_t(\rho^!\mathrm{gr}(\pi))=  \exp(-t\mathcal{L}_Z) \tau_{B_t}(\rho^!\mathrm{gr}(\pi)) = \rho^!\mathrm{gr}(\pi),
    $$
    which is equivalent to $\tau_{B_t}(\rho^!\mathrm{gr}(\pi))=\exp(t\mathcal{L}_Z)\rho^!\mathrm{gr}(\pi)$. Setting $t=1$, we obtain the condition
    $$
    \tau_{-\omega}(\rho^! \mathrm{gr}(\pi))= \exp(\mathcal{L}_Z)\rho^!\mathrm{gr}(\pi),
    $$
    for $\omega = -B_1$. The result now follows from Lemma~\ref{lem:bimodulecriterion}.
\end{proof}

Note that Lemma~\ref{lem:selfequiv} is a direct consequence of this proposition since the explicit formula for $\omega$ implies that $\omega = \omega_\canonical + \sum_{=1}^\infty \lambda^k d \theta_k$, for $\theta_k = -\int_0^1 \frac{(s\mathcal{L}_Z)^k}{k!} \theta_\canonical ds$.

\bibliographystyle{amsplain}

\def\cprime{$'$}
\providecommand{\bysame}{\leavevmode\hbox to3em{\hrulefill}\thinspace}
\providecommand{\MR}{\relax\ifhmode\unskip\space\fi MR }
\providecommand{\MRhref}[2]{%
  \href{http://www.ams.org/mathscinet-getitem?mr=#1}{#2}
}
\providecommand{\href}[2]{#2}

%
%


%
%

\end{document}
